\documentclass[preprint,12pt]{elsarticle}

\usepackage{amssymb}
\usepackage{amsmath}
\usepackage{subcaption}
\usepackage{multirow}

\newproof{proof}{Proof}
\newtheorem{theorem}{Theorem}
\newdefinition{remark}{Remark}
\journal{}
\makeatletter
\def\ps@pprintTitle{%
	\let\@oddhead\@empty
	\let\@evenhead\@empty
	\def\@oddfoot{}%
	\let\@evenfoot\@oddfoot}
\makeatother
\begin{document}
	
	\begin{frontmatter}
		
		
		
		\title{PCELM: Perturbation-Correction Extreme Learning Machine for the Stefan problem}
		
		\cortext[cor1]{Corresponding author.}
		
		\author[a]{Wenjie Liu}

		\author[a]{Siyuan Lang\corref{cor1}}
		\ead{2679040392@qq.com} 
		
		\author[a]{Zhiyue Zhang\corref{cor1}}
		\ead{zhangzhiyue@njnu.edu.cn} 
		\affiliation[a]{
			organization={ Ministry of Education Key Laboratory of NSLSCS, School of Mathematical Sciences}, 
			addressline={Nanjing Normal University}, 
			city={Nanjing},
			postcode={210023}, 
			state={Jiangsu Province},
			country={China}
		}
		
		\begin{abstract}
			For Stefan problems, characterized by moving boundaries and discontinuous coefficients due to phase changes, the inherent nonconvexity of the objective functional frequently causes optimization difficulty in randomized neural network approximations; to address this, we propose a Perturbation-Correction Extreme Learning Machine (PCELM) framework, built upon the extreme learning machine framework. This method first establishes a basic approximation during an initialization step by minimizing the original nonconvex residual, typically achieving only moderate accuracy, and then, in a subsequent correction step, determines a correction term by solving a subproblem based on a perturbation expansion around this basic approximation, thereby transforming it into a convex optimization problem for the output coefficients that ensures rapid convergence. We further provide a rigorous a convexity analysis, demonstrating that PCELM method solves a convex sub-problem. Numerical experiments on various Stefan problems, including multi-phase and multi-dimensional Stefan problems, confirm that the proposed PCELM method successfully overcomes optimization plateaus, with the correction step consistently delivering a significant improvement of 2-6 orders of magnitude in the relative $L^2$ accuracy.
		\end{abstract}
		
		
	\end{frontmatter}
	
	\section{Introduction}
	\par The Stefan problem is widely regarded as a prototypical mathematical model for phase-transition systems with moving boundaries, such as melting and solidification processes \cite{1,2,3,4,L,E}. It involves a time-evolving free boundary that partitions the computational domain into two distinct phases (e.g., solid and liquid). The dynamic nature of the interface, coupled with its interaction with the phase fields, poses significant challenges for numerical solvers. To date, both mesh-based and mesh-free methods have demonstrated distinct advantages in tackling the Stefan problem.
	
	\par Conventional mesh-based numerical methods for the Stefan problem may generally be grouped according to whether the moving interface is treated in an explicit or an implicit manner. In front-tracking approaches \cite{5,6,7,8,9}, the free boundary is directly followed by means of moving grids or Lagrangian markers; however, such techniques may become less effective when the interface undergoes complicated topological changes. By comparison, fixed-grid schemes \cite{10,11} avoid explicit interface tracking and instead capture the boundary motion on a stationary mesh. Within this category, the enthalpy method \cite{12,13,14} rewrites the governing model through an enthalpy-based formulation in which latent heat is absorbed into the source contribution. Another widely used strategy is the level-set method \cite{15,16,17}, where the interface is characterized by the zero level of an auxiliary higher-dimensional function. In addition, phase-field methods \cite{18,19,20,21} replace the sharp interface with a thin diffuse transition region.
	
	\par In recent years, mesh-free approaches have attracted increasing interest, largely motivated by the Universal Approximation Theorem \cite{22}. One notable example is the Physics-Informed Neural Network (PINN) paradigm \cite{23}, in which approximate solutions are constructed under the constraints of the governing equations together with the prescribed initial and boundary data. Relative to conventional analytical and numerical methods, PINN-based mesh-free techniques offer the potential to reduce computational expense for both forward and inverse problems. For the Stefan problem, a number of PINN-related studies \cite{24,25,26} have introduced the moving interface as an additional unknown so that it can be determined simultaneously with the phase variable. Moreover, the combination of level-set formulations and neural networks \cite{LF,MH} has also been explored as an effective strategy for free-boundary problems. A different research direction casts the partial differential equations (PDEs) into a least-squares formulation and computes the solution by means of the Random Feature Method (RFM) \cite{27,28,SZ2} or the Extreme Learning Machine (ELM) \cite{HG}. Based on this idea, Physics-Informed ELM (PIELM) approaches \cite{VB,SZ} have been introduced, where the parameters in the hidden layer are kept fixed and only the output-layer coefficients are trained. More recently, PIELM methods \cite{FP, CT, WY, AZ, SZ1} have been extended to capture evolving interfaces, demonstrating improved accuracy over standard PINN-based models. However, the availability of highly accurate, machine-learning-based solvers for the severe nonlinearities inherent in the Stefan problem remains limited, primarily due to the non-convex nature of the associated residual minimization. Developing efficient and robust optimization solvers in this context is therefore of paramount importance.
	
	\par In this work, we propose a perturbation-correction method utilizing the ELM framework for the efficient and robust solution of the Stefan problem. In the first stage, the original non-convex PDE residual minimization is solved via ELM, yielding a convergent base approximation that captures the primary dynamics of the interface and the temperature field. Subsequently, in the correction stage, we apply a perturbation expansion around the numerical solution obtained in the first step. This reformulates the correction step into a mathematically convex optimization problem with respect to the output coefficients. In both stages, the resulting optimization problems are efficiently minimized using Gauss--Newton iterations, ensuring rapid convergence and high precision. Furthermore, we provide a rigorous convexity analysis for the perturbation-correction stage. Finally, to demonstrate the effectiveness and robustness of our approach, we present numerical tests on a suite of representative benchmarks, including 1D one-phase and two-phase free boundary problems, 2D one-phase free boundary problem, as well as 2D and 3D Frank-sphere problems.
	
	The main contributions are as follows:
	
	\begin{enumerate}
		\item We develop a novel perturbation correction framework for the Stefan problem, which incorporates a free-boundary perturbation expansion step to effectively resolve optimization stagnation in the non-convex residual minimization process.
		
		\item A comprehensive convexity analysis for the correction step, proving that it formulates a convex sub-problem capable of systematically eliminating leading-order initialization errors.
		
		\item Extensive numerical experiments across challenging scenarios---including multi-dimensional, multi-phase Stefan problems, and high-dimensional Frank-sphere problems---demonstrating consistent error reduction between the two steps. The proposed method achieves at least a two-order-of-magnitude improvement in relative $L^2$ accuracy and exhibits stable performance across all tested configurations.
	\end{enumerate}
	
	\par The rest of this paper is arranged as follows. Section 2 presents the mathematical description of the Stefan problem. In Section 3, we introduce the ELM framework and develop the proposed perturbation-correction approach. Section 4 is concerned with the related convex analysis. Section 5 contains several numerical examples to demonstrate the performance of the method. Finally, Section 6 summarizes the main conclusions of this work.
	\section{Mathematical model}
	Let $\Omega \subset \mathbb{R}^d$ ($d \in \{1,2,3\}$) be a bounded domain partitioned into two subdomains, $\Omega_{1}(t)$ and $\Omega_{2}(t)$, separated by a moving interface $R(t)$. The two subdomains and the interface depend on time $t$ and can evolve. We denote $\partial \Omega_{q}(t)$ as the boundary of $\Omega_{q}(t)$ for $q=1, 2$. Let $R_{1}(t)=\partial \Omega_{1}(t)\backslash R(t)$ and $R_{2}(t)=\partial \Omega_{2}(t)\backslash R(t)$ represent the fixed boundaries. Denote $u(\mathbf{x}, t)$ as the scalar temperature field in the whole bounded domain $\Omega$, where $u_{q}(\boldsymbol{x}, t)$ is the temperature distribution in each subdomain $\Omega_{q}(t)$. The two-phase Stefan problem, which governs the evolution of the solid and liquid phases, is formulated as follows
	\begin{equation}
		\label{1}
		\frac{\partial u_{q}}{\partial t}(\mathbf{x}, t)-k_{q}\Delta u_{q}(\mathbf{x}, t)=0, \mathbf{x}\in \Omega_{q}(t), t\in (0, T], q=1,2,\\
	\end{equation}
	where $k_{1}$ and $k_{2}$ are the given parameters. The initial and boundary conditions are defined for the Stefan problem
	\begin{equation}\label{2}
		\begin{split}
			&u_{q}(\mathbf{x},0)=g_{q}(\mathbf{x}), \mathbf{x}\in \Omega_{q}(0), q=1,2,\\
			&u_{q}(\mathbf{x},t)=h_{q}(\mathbf{x},t), \mathbf{x}\in R_{q}(t), t\in (0,T], q=1,2.
		\end{split}
	\end{equation}
	The initial and boundary conditions where the free boundary also satisfies
	\begin{equation}\label{3}
		\begin{split}
			&R(0)=R_{0},\\
			&u_{1}(\mathbf{x},t)=u_{2}(\mathbf{x},t)=0, \mathbf{x}\in R(t), t\in (0,T],\\
			&k_{1}\frac{\partial u_{1}}{\partial \mathbf{n}}-k_{2}\frac{\partial u_{2}}{\partial \mathbf{n}}=h(\mathbf{x}, t), \mathbf{x}\in R(t), t\in (0, T],
		\end{split}
	\end{equation}
	where $\mathbf{n}$ signifies the unit normal vector of $R(t)$. Figure \ref{fig:11} illustrates two distinct evolution modes of the free boundary in the two-phase Stefan problem.
	\begin{figure} 
		\centering
		\begin{subfigure}[b]{0.45\textwidth}
			\centering
			\includegraphics[width=\textwidth]{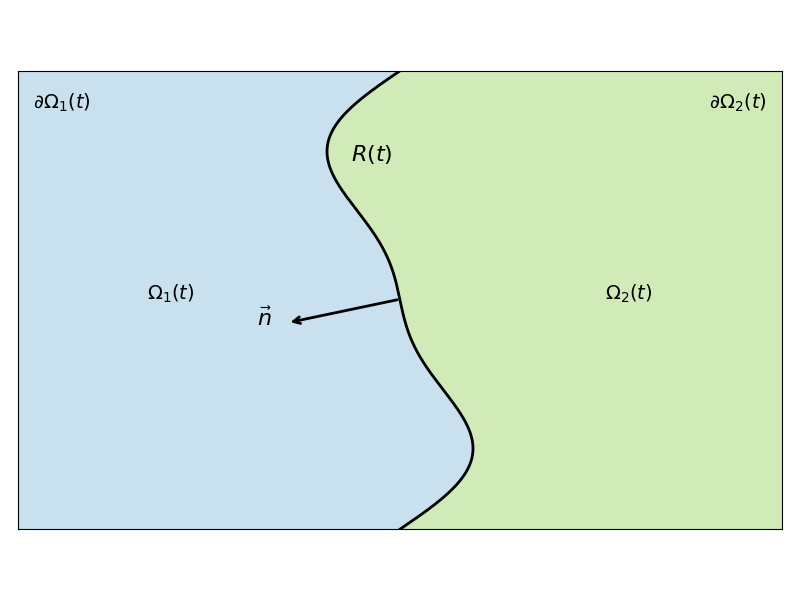} 
			\label{fig:sub11}
		\end{subfigure}
		\hfill
		\begin{subfigure}[b]{0.45\textwidth}
			\centering
			\includegraphics[width=\textwidth]{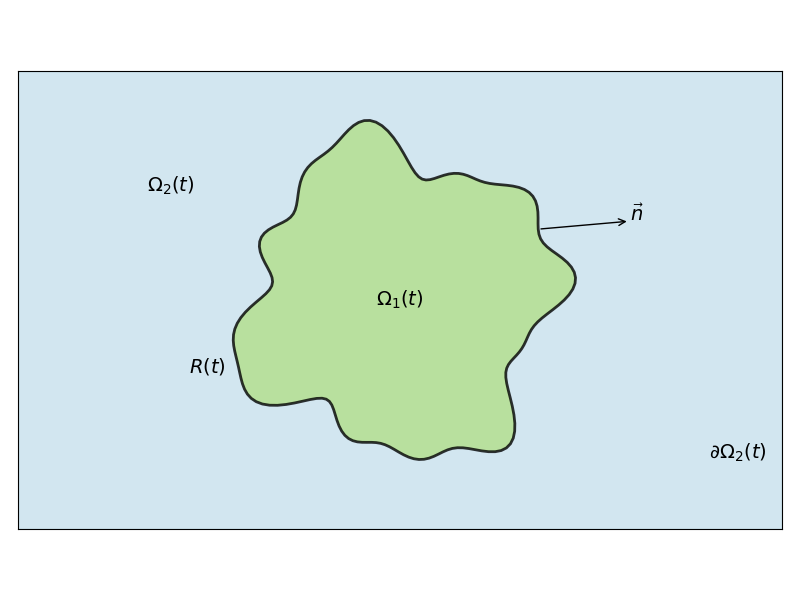} 
			\caption{}
			\label{fig:sub12}
		\end{subfigure}
		\caption{Schematic of the two-phase Stefan problem illustrating the evolution of the free boundary and two-phase distributions.}
		\label{fig:11}
	\end{figure}
	\section{PCELM framework for the Stefan problem}
	\label{sec2}
	Before employing PCELM method to solve the Stefan problem, we first need to understand the ELM framework. 
	\subsection{Extreme learning machine}
	The extreme learning machine, as a random feature method, has excellent function approximation properties. We consider a general form of partial
	differential equations
	\begin{equation*}
		\begin{split}
			&\frac{\partial u}{\partial t}(\mathbf{x}, t)-\mathcal{L}[u(\mathbf{x}, t)]=0, (\mathbf{x}, t)\in \Omega\times(0,T],\\
			&\mathcal{IC}[u(\mathbf{x}, 0)]=g(\mathbf{x}),x\in \Omega,\\
			&\mathcal{BC}[u(\mathbf{x}, t)]=h(\mathbf{x}, t), x\in \partial \Omega, t\in (0,T],
		\end{split}
	\end{equation*}
	where $\mathbf{x}\in \mathbb{R}^d$; $\mathcal{L}[\cdot]$ is a nonlinear differential operator; $\mathcal{IC}[\cdot]$ and $\mathcal{BC}[\cdot]$ denote the initial condition and the condition operator separately. In the ELM framework, the solution can be parameterized by a single-hidden-layer network. A schematic of the neural network architecture is presented in Figure \ref{fig:1111}.
	\begin{figure}[htbp] 
		\centering
		\includegraphics[width=\textwidth]{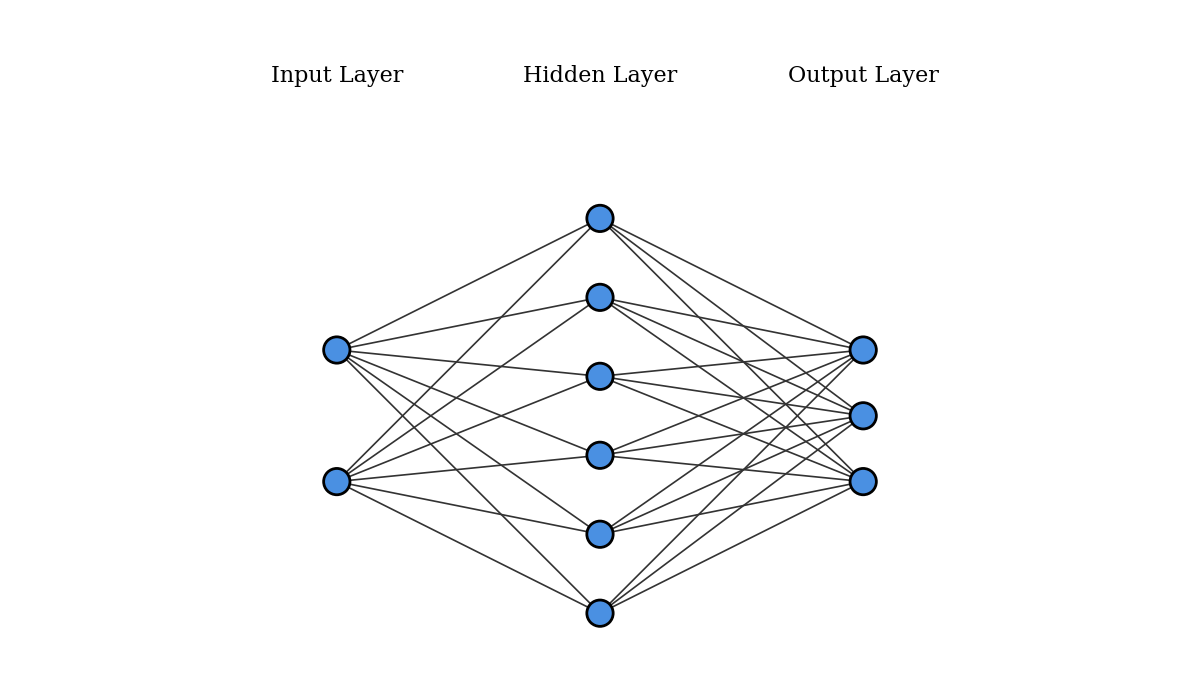}
		\caption{Structure of the extreme learning machine.}
		\label{fig:1111} 
	\end{figure}
	\par We denote $u^{M}(\mathbf{x},t)$ as the approximate solution, constructed by $M$ linear basis functions
	\begin{equation*}
		u^{M}(\mathbf{x},t;\boldsymbol{\alpha})=\sum_{j=1}^{M}{\alpha_{j}}\tau\left(\boldsymbol{\omega_{j}}\cdot\mathbf{x}+\boldsymbol{\omega_{j}}\cdot t+b_{j}\right)=:\boldsymbol{\tau}(\mathbf{x},t)\boldsymbol{\boldsymbol{\alpha}},
	\end{equation*}
	where $\boldsymbol{\alpha}=\left[\alpha_{1},\cdots,\alpha_{M}\right]$ denotes the unknown coefficients, while $\boldsymbol{\omega_{j}} \in \mathbb{R}^d$ and $b_{j}\in \mathbb{R}$ are generated by a probability distribution and remain fixed in the computation. Furthermore, $\tau_{j}(\mathbf{x}, t)=\tau (\boldsymbol{\omega_{j}}\cdot\mathbf{x}+\boldsymbol{\omega_{j}}\cdot t+b_{j})$ indicates $\boldsymbol{\tau}(\mathbf{x},t)=\left[\tau_{1}(\mathbf{x}, t),\cdots,\tau_{M}(\mathbf{x}, t)\right]$. 
	\par The loss function is set for measuring the distance between the output $u(\mathbf{x}, t; \boldsymbol{\alpha})$ and the reference solution. We optimize the parameter $\boldsymbol{\alpha}$ by minimizing the loss function which follows
	\begin{equation*}
		Loss(\boldsymbol{\alpha})=\beta_{e}Loss_{e}(\boldsymbol{\alpha})+\beta_{ic}Loss_{ic}(\boldsymbol{\alpha})+\beta_{bc}Loss_{bc}(\boldsymbol{\alpha}),
	\end{equation*}
	with
	\begin{equation*}
		\begin{split}
			&Loss_{e}(\boldsymbol{\alpha})=\frac{1}{N_{e}}\sum_{i=1}^{N_{e}}\left|\frac{\partial u}{\partial t}(\mathbf{x}_{i}, t_{i}; \boldsymbol{\alpha})-\mathcal{L}[u(\mathbf{x}_{i}, t_{i}; \boldsymbol{\alpha})] \right|^{2},\\
			&Loss_{ic}(\boldsymbol{\alpha})=\frac{1}{N_{ic}}\sum_{k=1}^{N_{ic}}\left|\mathcal{IC}[u(\mathbf{x_{k}}, 0;\boldsymbol{\alpha})]-g(\mathbf{x_{k}};\boldsymbol{\alpha})\right|^{2},\\
			&Loss_{bc}(\boldsymbol{\alpha})=\frac{1}{N_{bc}}\sum_{l=1}^{N_{bc}}\left|\mathcal{BC}[u(\mathbf{x_{l}}, 0;\boldsymbol{\alpha})]-h(\mathbf{x_{l}}; \boldsymbol{\alpha})\right|^{2},
		\end{split}
	\end{equation*}
	where $\beta_{e}$, $\beta_{ic}$ and $\beta_{bc}$ denote the weight coefficients; $N_{e}$, $N_{ic}$ and $N_{bc}$ are the points for the equation, the initial condition and the boundary condition, respectively. $Loss_{e}$, $Loss_{ic}$ and $Loss_{bc}$ measure the equation residual with randomly sampled points $(\mathbf{x}_{i}, t_{i})^{N_{e}}_{i=1}$ on the bounded domain $\Omega\times (0, T]$, the initial residual with $(\mathbf{x}_{k}, t_{k})^{N_{ic}}_{k=1}$ on $\Omega$ and the boundary residual with $(\mathbf{x_{l}}, t_{l})^{N_{bc}}_{l=1}$ on $\partial \Omega\times (0\times T]$, respectively. Activation function has many choices, such as tanh, ReLU, etc.  The output parameter $\boldsymbol{\alpha}$ incorporated into the loss functions can be trained by using Adam or SGD methods to realize the minimization of the total loss.
	\subsection{ELM method for the Stefan problem}
	\label{s22}
	In this subsection, we present the ELM framework for solving the two-phase Stefan problem. Unlike standard partial differential equations (PDEs), the Stefan problem involves an unknown evolving interface. Therefore, we let 
	\begin{equation}
		u^{M}_{q}(\mathbf{x}, t; \boldsymbol{\alpha}_{q})=\sum_{i=1}^{M}\alpha_{i,q}\tau^{u} \left(\boldsymbol{\omega}_{i}\cdot \boldsymbol{x}+\boldsymbol{\omega}_{i}\cdot t+ b_{i,q} \right)=:\boldsymbol{\tau^{u}}(\mathbf{x},t)\boldsymbol{\alpha}_{q},
	\end{equation} 
	which denotes the approximation of the $q$-th phase field $u_{q}(\mathbf{x}, t)$ for $q \in \{1, 2\}$, and we approximate the exact free boundary $R(t)$ using a parameterized function 
	\begin{equation}
		R^{H}(t; \boldsymbol{\gamma})=\sum_{j=1}^{H}\gamma_{j} \tau^{R}\left(\boldsymbol{\omega}_{j}\cdot t +b_{j} \right)=:\boldsymbol{\tau^{R}}(t)\boldsymbol{\gamma}.
	\end{equation} 
	\par Here, $\boldsymbol{\alpha}_{q}$ and $\boldsymbol{\gamma}$ represent independent trainable parameters. We select $\sin$ as the activation function because it belongs to a Fourier-type function space capable of approximating smooth functions. The parameters $\boldsymbol{\alpha}_{q}$ and $\boldsymbol{\gamma}$ are optimized by minimizing the total loss function derived from Eqs. (\ref{1})-(\ref{3})
	\begin{equation}
		\begin{split}
			&Loss\left(\boldsymbol{\alpha}_{1},\boldsymbol{\alpha}_{2},\boldsymbol{\gamma}\right)\\
			&=\sum_{q=1}^{2}\Big[\beta_{e, q}Loss_{e, q}(\boldsymbol{\alpha}_{q})+\beta_{ic, q}Loss_{ic, q}(\boldsymbol{\alpha}_{q})+\beta_{bc, q}Loss_{bc, q}(\boldsymbol{\alpha}_{q})\\
			&+\beta_{bs, q}Loss_{bs, q}(\boldsymbol{\alpha}_{q}, \boldsymbol{\gamma}) \Big]+\beta_{fb}Loss_{fb}(\boldsymbol{\gamma})+\beta_{sc}Loss_{sc}(\boldsymbol{\alpha}_{1},\boldsymbol{\alpha}_{2}, \boldsymbol{\gamma}),
		\end{split}
	\end{equation}
	where $\beta_{e,q}$, $\beta_{ic,q}$, $\beta_{bc,q}$, $\beta_{bs,q}$, $\beta_{fb}$ and $\beta_{sc}$ represent the positive weights for balancing the different loss terms.
	\par To clearly illustrate the framework, the individual loss components are defined sequentially below.
	
	\par For $q=1, 2$, the loss term corresponding to the governing equation residual is given by:
	\begin{equation}
		Loss_{e, q}(\boldsymbol{\alpha}_{q})=\frac{1}{N_{e}}\sum_{i=1}^{N_{e}}\left|\frac{\partial u_{q}}{\partial t}(\mathbf{x}_{i}, t_{i}; \boldsymbol{\alpha}_{q})-k_{q}\Delta u_{q}(\mathbf{x}_{i}, t_{i}; \boldsymbol{\alpha}_{q}) \right|^{2}:=\frac{1}{N_{e}}\sum_{i=1}^{N_{e}}\left|\mathcal{R}_{e}\right|^{2}.
	\end{equation}
	
	The loss term for the initial condition of the phase fields is formulated as:
	\begin{equation}
		Loss_{ic,q}(\boldsymbol{\alpha}_{q})=\frac{1}{N_{ic}}\sum_{k=1}^{N_{ic}}\left|u_{q}(\mathbf{x}_{k}, 0; \boldsymbol{\alpha}_{q})-g_{q}(\mathbf{x}_{k})\right|^{2}:=\frac{1}{N_{ic}}\sum_{k=1}^{N_{ic}}\left|\mathcal{R}_{ic}\right|^{2}.
	\end{equation}
	
	The loss term enforcing the outer boundary condition is defined as:
	\begin{equation}
		Loss_{bc,q}(\boldsymbol{\alpha}_{q})=\frac{1}{N_{bc}}\sum_{l=1}^{N_{bc}}\left|u_{q}(\mathbf{x}_{l}, t_{l};\boldsymbol{\alpha}_{q})-h_{q}(\mathbf{x}_{l},t_{l}) \right|^2:=\frac{1}{N_{bc}}\sum_{l=1}^{N_{bc}}\left|\mathcal{R}_{bc}\right|^{2}.
	\end{equation}
	
	The loss term for the condition evaluated strictly on the moving boundary is:
	\begin{equation}
		Loss_{bs,q}(\boldsymbol{\alpha}_{q}, \boldsymbol{\gamma})=\frac{1}{N_{bs}}\sum_{p=1}^{N_{bs}}\left|u_{q}(R(t_{p}; \boldsymbol{\gamma}), t_{p}; \boldsymbol{\alpha}_{q})\right|^2:=\frac{1}{N_{bs}}\sum_{p=1}^{N_{bs}}\left|\mathcal{R}_{bs}\right|^{2}.
	\end{equation}
	
	Specifically for the moving boundary, the loss term constraining its initial position is given by:
	\begin{equation}
		Loss_{fb}(\boldsymbol{\gamma})=\frac{1}{N_{fb}}\sum_{i=1}^{N_{fb}}\left|R(t_{i}; \boldsymbol{\gamma})-R_{0} \right|^2:=\frac{1}{N_{fb}}\sum_{i=1}^{N_{fb}}\left|\mathcal{R}_{fb}\right|^{2}.
	\end{equation}
	
	Finally, the loss term for the Stefan condition, which couples the two phase fields at the moving boundary, is defined as:
	\begin{equation}
		\begin{split}
			&Loss_{sc}(\boldsymbol{\alpha}_{1},\boldsymbol{\alpha}_{2}, \boldsymbol{\gamma})=\\
			&\frac{1}{N_{sc}}\sum_{j=1}^{N_{sc}}\left|k_{1}\frac{\partial u_{1}}{\partial \mathbf{n}}\left(R(t_{j}; \boldsymbol{\gamma}),t_{j}; \boldsymbol{\alpha}_{1}\right)-k_{2}\frac{\partial u_{2}}{\partial \mathbf{n}}\left(R(t_{j};\boldsymbol{\gamma}),t_{j};\boldsymbol{\alpha}_{2}\right)-h(\mathbf{x}_{j},t_{j})\right|^2\\
			&:=\frac{1}{N_{sc}}\sum_{j=1}^{N_{sc}}\left|\mathcal{R}_{sc}\right|^{2}.
		\end{split}
	\end{equation}
	\par Here, $N_{e}$, $N_{ic}$, $N_{bc}$, $N_{bs}$, $N_{fb}$, and $N_{sc}$ denote the number of collocation points sampled for the governing equations, initial conditions, boundary conditions, moving interface condition, initial status of the free boundary, and the Stefan condition, respectively. $\mathcal R_{e}$, $\mathcal R_{ic}$, $\mathcal R_{bc}$, $\mathcal R_{bs}$, $\mathcal R_{fb}$, and $\mathcal R_{sc}$ represent the corresponding point-wise residuals for each equation.
	\par Considering that the total loss functional takes the form of a discrete least-squares problem, we employ a regularized Gauss-Newton scheme for its minimization
	\begin{equation}
		\min_{\boldsymbol{\eta}} G(\boldsymbol{\eta}) = \min_{\boldsymbol{\eta}} Loss \left(\boldsymbol{\eta}\right), \label{eq:1}
	\end{equation}
	where $\boldsymbol{\eta} = \left(\boldsymbol{\alpha}_{1}, \boldsymbol{\alpha}_{2}, \boldsymbol{\gamma}\right)$ is the concatenated trainable parameter vector. By forming a block column, the global residual vector $\mathcal{T}$ is formulated as
	\begin{equation}
		\mathcal{T} = \begin{pmatrix} 
			\mathcal{R}_{e} \\ 
			\mathcal{R}_{ic} \\ 
			\mathcal{R}_{bc} \\ 
			\mathcal{R}_{bs} \\ 
			\mathcal{R}_{fb} \\
			\mathcal{R}_{sc}
		\end{pmatrix}, 
	\end{equation}
	with $\mathcal{T} \in \mathbb{R}^{N_{total}}$, where $N_{total} = N_{e} + N_{ic} + N_{bc} + N_{bs} + N_{fb} + N_{sc}$.
	
	To solve this nonlinear least-squares problem Eq. (\ref{eq:1}), we implement a regularized Gauss-Newton scheme. At iteration $k$, we approximate the Hessian matrix as $\nabla^2 G^{\{k\}} \approx {\mathcal{J}^{\{k\}}}^T \mathcal{J}^{\{k\}}$, where $\mathcal{J}^{\{k\}} = \partial \mathcal{T}^{\{k\}} / \partial \boldsymbol{\eta}$ is the associated Jacobian matrix, given by
	\begin{equation}
		\mathcal{J} = \begin{pmatrix} 
			\mathcal{J}_{e} & \mathbf{0} \\ 
			\mathcal{J}_{ic} & \mathbf{0} \\ 
			\mathcal{J}_{bc} & \mathbf{0} \\ 
			\mathcal{J}_{bs,u} & \mathcal{J}_{bs,R} \\ 
			\mathbf{0} & \mathcal{J}_{fb,R} \\ 
			\mathcal{J}_{sc, u} & \mathcal{J}_{sc,R} 
		\end{pmatrix}. \label{eq:2}
	\end{equation}
	The $\mathbf{0}$ block matrices indicate that specific residual terms are strictly independent of the free boundary or phase field parameters, and therefore their corresponding partial derivatives vanish identically.
	
	\par Let $(\mathbf{x}_i, t_i)$ denote the respective collocation points. For the activation function is $\tau_j^{u}$, the Jacobian submatrix $\mathcal{J}_{e}$ corresponding to the governing equation is analytically evaluated as
	\begin{equation}
		(\mathcal{J}_{e})_{i,j} := \left. \frac{\partial \mathcal{R}_{e}}{\partial \alpha_{q, j}} \right|_{(\mathbf{x}_i, t_i)} = \frac{\partial \tau_j^{u}}{\partial t} \Big|_{(\mathbf{x}_i, t_i)} - k_{q} \Delta \tau_j^{u} \Big|_{(\mathbf{x}_i, t_i)}. \label{eq:3}
	\end{equation}
	\par Then, the submatrices $\mathcal{J}_{ic}$ and $\mathcal{J}_{bc}$, associated with the initial and boundary conditions of the phase fields, are defined respectively by
	\begin{equation}
		(\mathcal{J}_{ic})_{i,j} := \left. \frac{\partial \mathcal{R}_{ic}}{\partial \alpha_{q, j}} \right|_{(\mathbf{x}_i, 0)} = \tau_j^{u} \Big|_{(\mathbf{x}_i, 0)}, \label{eq:4}
	\end{equation}
	\begin{equation}
		(\mathcal{J}_{bc})_{i,j} := \left. \frac{\partial \mathcal{R}_{bc}}{\partial \alpha_{q, j}} \right|_{(\mathbf{x}_i, t_i)} = \tau_{j}^{u} \Big|_{(\mathbf{x}_i, t_i)}. \label{eq:5}
	\end{equation}
	
	\par At the moving boundary, we define the submatrices $\mathcal{J}_{bs,u}$ and $\mathcal{J}_{bs,R}$ for the phase evaluation evaluation with respect to parameters $\boldsymbol{\alpha}$ and $\boldsymbol{\gamma}$ as
	\begin{equation}
		\begin{split}
			(\mathcal{J}_{bs,u})_{i,j} :=& \left. \frac{\partial \mathcal{R}_{bs}}{\partial \alpha_{q, j}} \right|_{(\mathbf{x}_i, t_i)} = \tau_j^{u} \Big|_{(\mathbf{x}_i, t_i)}, \\
			(\mathcal{J}_{bs,R})_{i,j} :=& \left. \frac{\partial \mathcal{R}_{bs}}{\partial \gamma_j} \right|_{(\mathbf{x}_i, t_i)} = \left( \nabla u_{q} \cdot \tau_j^{R} \right) \Big|_{(\mathbf{x}_i, t_i)}.
		\end{split}
	\end{equation}
	\par For the initial condition of the localized free boundary, $\mathcal{J}_{fb,R}$ is formulated as
	\begin{equation}
		(\mathcal{J}_{fb,R})_{i,j} := \left. \frac{\partial \mathcal{R}_{fb}}{\partial \gamma_j} \right|_{(\mathbf{x}_i, t_i)} = \tau_j^{R} \Big|_{(\mathbf{x}_i, t_i)}. 
	\end{equation}
	
	\par For the Stefan condition linking both phase fields and the free boundary, $\mathcal{J}_{sc,u}$ and $\mathcal{J}_{sc,R}$ are derived as
	\begin{equation}
		\begin{split}
			(\mathcal{J}_{sc,u})_{i,j} :=& \left. \frac{\partial \mathcal{R}_{sc}}{\partial \alpha_{q, j}} \right|_{(\mathbf{x}_i, t_i)} = (-1)^{q+1}k_{q} \frac{\partial \tau_j^{u}}{\partial \mathbf{n}} \Big|_{(\mathbf{x}_i, t_i)}, \\
			(\mathcal{J}_{sc,R})_{i,j} :=& \left. \frac{\partial \mathcal{R}_{sc}}{\partial \gamma_j} \right|_{(\mathbf{x}_i, t_i)} = \left( k_{1} \frac{\partial^2 u_{1}}{\partial \mathbf{n} \partial \mathbf{x}} -k_{2} \frac{\partial^2 u_{2}}{\partial \mathbf{n} \partial \mathbf{x}} \right) \cdot \tau_j^{R} \Big|_{(\mathbf{x}_i, t_i)}.
		\end{split}
	\end{equation}
	
	The parameter update $\delta\boldsymbol{\eta} = (\delta\boldsymbol{\alpha}_{1}, \delta\boldsymbol{\alpha}_{2}, \delta\boldsymbol{\gamma})$ is calculated by solving the linear system
	\begin{equation}
		{\mathcal{J}^{\{k\}}}^T \mathcal{J}^{\{k\}} \delta\boldsymbol{\eta} = -{\mathcal{J}^{\{k\}}}^T \mathcal{T}^{\{k\}}, \label{eq:6}
	\end{equation}
	where $k$ denotes the current iteration step. To guarantee a stable and effective matrix inversion, $\delta\boldsymbol{\eta}$ is computed via the singular value decomposition (SVD) by utilizing the Moore--Penrose pseudo-inverse ${\mathcal{J}^{\{k\}}}^\dagger$. Let the SVD of $\mathcal{J}^{\{k\}} \in \mathbb{R}^{N_{total} \times M_{\eta}}$ (where $M_{\eta}$ represents the total dimension of $\boldsymbol{\eta}$) be expressed as
	\begin{equation}
		\mathcal{J}^{\{k\}} = U \Sigma V^T, \label{eq:20}
	\end{equation}
	where $U \in \mathbb{R}^{N_{total} \times r}$ and $V \in \mathbb{R}^{M_{\eta} \times r}$ are column-orthonormal matrices, and $\Sigma = \text{diag}\{\nu_1, \nu_2, \cdots, \nu_r, 0, \cdots\}$ contains the non-zero singular values strictly satisfying $\nu_1 \ge \nu_2 \ge \cdots \ge \nu_r > 0$, with $r = \text{rank}(\mathcal{J}^{\{k\}})$. While constructing the pseudo-inverse, we truncate trailing singular values smaller than an assigned threshold $\rho > 0$, keeping only the index set $\mathcal{I}_r = \{i : \nu_i \ge \rho\}$. The truncated pseudo-inverse of $\Sigma$ becomes
	\begin{equation}
		(\Sigma_r^\dagger)_{ii} = \begin{cases} 
			1/\nu_i, & \nu_i \ge \rho, \\ 
			0, & \nu_i < \rho. 
		\end{cases} \label{eq:7}
	\end{equation}
	
	Taking this into consideration, the regularized update is naturally deduced as
	\begin{equation}
		\delta\boldsymbol{\eta} = -V \Sigma_r^\dagger U^T \mathcal{T}^{\{k\}}. \label{eq:8}
	\end{equation}
	Thereby, the network parameter vector $\boldsymbol{\eta}$ is systematically advanced following a standard Newton step
	\begin{equation}
		\boldsymbol{\eta}^{\{k+1\}} = \boldsymbol{\eta}^{\{k\}} + \delta\boldsymbol{\eta}. \label{eq:9}
	\end{equation}
	
	This iteration loop terminates smoothly once the relative variation of the total loss function fulfills the given tolerance criterion
	\begin{equation}
		\big|Loss(\boldsymbol{\eta}^{\{k+1\}}) - Loss(\boldsymbol{\eta}^{\{k\}}) \big| \le \delta_0 \big |Loss(\boldsymbol{\eta}^{\{k+1\}})\big|. \label{eq:10}
	\end{equation}
	
	Here, $\boldsymbol{\eta}_0$ sets a predefined initial state, typically defined as $\boldsymbol{\eta}_0 = \mathbf{0}$. 
	
	Crucially, the inherent mathematical equivalence between solving the nonlinear residual vector form $\mathcal{T}(\boldsymbol{\eta}) = \mathbf{0}$ and minimizing the scalar functional $G(\boldsymbol{\eta})$ is rigorously guaranteed. This stems from the analytical relationship $\nabla_{\boldsymbol{\eta}} G = {\mathcal{J}}^T \mathcal{T}$, further substantiating that ${\mathcal{J}}^T \mathcal{J}$ serves as a reliable approximation to the local Hessian matrix. Such an established identity validates that the right-hand side term ${\mathcal{J}}^T \mathcal{T}$ of Eq. (\ref{eq:6}) authentically represents the gradient of $G$ relative to $\boldsymbol{\eta}$. Nonetheless, owing strictly to the highly non-convex nature of $G(\boldsymbol{\eta})$, this optimization scheme might eventually fall into local minimum for the parameters $\left(\boldsymbol{\alpha}_{1}, \boldsymbol{\alpha}_{2}, \boldsymbol{\gamma}\right)$.
	\subsection{PCELM Method}
	To further reduce the residuals remaining after the ELM pre-training step, we employ a perturbation-based correction strategy designed to exploit a locally quadratic approximation of the residual functional. Since the highly nonlinear and nonconvex nature of problem Eqs. (\ref{1}-\ref{3}) can cause the Gauss--Newton iteration to converge to a local minimizer $u_{q}^{M}~(q=1, 2)$ and $R^{H}$ with nonzero residuals, we seek to construct a refined solution of the form
	\begin{equation}
		\begin{split}
			&u_q^g =  u_q^M+ \epsilon u_q^{pc},\qquad \epsilon u_q^{pc} \approx u_q^g - u_q^M = e_{u,q},\quad q=1,2,\\
			&R^{g} = R^{H} + \epsilon R^{pc},\qquad \epsilon R^{pc} \approx R^{g} - R^{H}= e_{R},
		\end{split}
	\end{equation}
	where $\epsilon > 0$ is a small parameter proportional to $|Loss(u_{1}^{M},u_{2}^{M},R^{H})|$, and $u_q^{pc}$ and $R^{pc}$ represent the perturbation terms for the temperature field and moving boundary, respectively. We set the perturbation parameter $\epsilon = |Loss(u_{1}^{M},u_{2}^{M},R^{H})|$ to ensure numerical stability and proper scaling of the correction term. This choice ensures that the unknown perturbations $u_{1}^{pc}$, $u_{2}^{pc}$, $R^{pc}$ remain $\mathcal{O}(1)$ in magnitude. By solving the perturbation subproblem for $u_{1}^{pc}, u_{2}^{pc}, R^{pc}$ to high accuracy, the constructed corrections $\epsilon u_{1}^{pc}, \epsilon u_{2}^{pc}, \epsilon R^{pc}$ yield high-precision estimations of the true errors $e_{u,1}, e_{u,2}, e_{R}$, respectively.
	\par It should be emphasized that, in terms of the structure of the governing partial differential equations, the derived perturbation system is still of Stefan type. The essential nonlinear interaction linking the phase interface to the moving boundary is retained from the original formulation, although the source terms are replaced by modified expressions obtained through the Taylor expansion at $u_{1}^{M}$, $u_{2}^{M}$, and $R^{H}$. Therefore, the ELM-based framework developed for the original model can be naturally extended to this perturbation problem.
	\par More precisely, $u_{1}^{pc}$, $u_{2}^{pc}$, and $R^{pc}$ are represented individually by single-hidden-layer feedforward neural networks, where the hidden-layer weights and biases are assigned randomly and kept fixed, while only the coefficients in the output layer are determined through optimization
	\begin{equation}
		\begin{split}
			&u_q^{pc} = \sum_{j=1}^M \lambda_{j,q}\tau^{u}\left(\mathbf{w}^{pc}_{j}\cdot \mathbf{x}+\mathbf{w}^{pc}_{j}\cdot t+ b^{pc}_{j,q} \right),\quad q=1,2, \\
			&R^{pc} = \sum_{j=1}^H \theta_{j} \tau^{R}(\mathbf{w}^{pc}_{j} \cdot t + b^{pc}_{j}),
		\end{split}
	\end{equation}
	where $\tau^{u}$ and $\tau^{R}$ are nonlinear activation functions introduced for enhanced expressivity, remaining distinct from those utilized in $u_{1}^{M}, u_{2}^{M}$, and $R^{H}$. The parameters $\mathbf{w}^{pc}_{j}$ and $b^{pc}_{j,q}, b^{pc}_{j}$ are randomly initialized and held fixed. We adopt $\sin$ as the uniform activation function for both the primary ELM framework and the correction step, as it demonstrates superior capability in capturing high-frequency oscillations compared to standard hyperbolic tangents.
	\par The parameter vector $\boldsymbol{\xi} = (\boldsymbol{\lambda_{1}}, \boldsymbol{\lambda_{2}}, \boldsymbol{\theta})$ is optimized by minimizing the following least-squares problem
	\begin{equation} \label{eq:gamma_opt}
		\min_{\boldsymbol{\xi}} G_p(\boldsymbol{\xi}) = \min_{\boldsymbol{\xi}} \left\| \mathcal{T}_p \right\|^2_2,
	\end{equation}
	where $\mathcal{T}_p$ constitutes the discretized residual vector for the perturbation formulation
	\begin{equation}
		\mathcal{T}_p = \frac{1}{\epsilon} 
		\begin{pmatrix}
			\mathcal{R}_{e,q}^{pc} \\
			\mathcal{R}_{ic,q}^{pc} \\
			\mathcal{R}_{bc,q}^{pc} \\
			\mathcal{R}_{bs,q}^{pc} \\
			\mathcal{R}_{fb}^{pc}\\
			\mathcal{R}_{sc}^{pc}
		\end{pmatrix},
	\end{equation}
	which is evaluated at the respective collocation point sets $N_{e}^{pc}$, $N_{ic}^{pc}$, $N_{bc}^{pc}$, $N_{bs}^{pc}$, $N_{fb}^{pc}$ and $N_{sc}^{pc}$.
	\par Explicitly, for each interior collocation point $(\mathbf{x}_i,t_i) \in N_{e}^{pc}$, the pointwise-defined perturbation residual $\mathcal{R}_{e,q}^{pc}$ is derived by substituting the perturbation ansatz $u_q^g = u_q^M + \epsilon u_q^{pc}$ into the governing equation (\ref{1}). Expanding in powers of $\epsilon$ and invoking the linearity of the PDE operator $L_q(u) = \frac{\partial u}{\partial t} - k_q \Delta u$, the expansion terminates exactly at $\mathcal{O}(\epsilon)$
	\begin{equation}
		\mathcal{R}_{e,q}^{pc}(\mathbf{x}_i, t_i) = \epsilon^0 \left( \frac{\partial u_q^M}{\partial t} - k_q \Delta u_q^M \right) + \epsilon^1 \left( \frac{\partial u_q^{pc}}{\partial t} - k_q \Delta u_q^{pc} \right).
	\end{equation}
	The corresponding Jacobian matrix entries evaluated at $(\mathbf{x}_i, t_i)$ are obtained by differentiating the residual with respect to the continuous parameters $\lambda_{j,q}$
	\begin{equation}
		(\mathcal{J}^{pc}_{e,q})_{i, j} = \left. \frac{\partial \mathcal{R}_{e,q}^{pc}}{\partial \lambda_{j,q}} \right|_{(\mathbf{x}_i, t_i)} = \epsilon \left( \frac{\partial \tau^{u}_{j,q}}{\partial t} - k_q \Delta \tau^{u}_{j,q} \right).
	\end{equation}
	Evidently, the Jacobian solely receives contributions from the $\mathcal{O}(\epsilon)$ term, and the $\mathcal{O}(\epsilon^2)$ contribution guarantees to be zero. 
	
	Analogously, for the initial and fixed boundary points residing in $N_{ic}^{pc}$ and $N_{bc}^{pc}$, we deduce the residuals by substituting $u_q^g$ into Eq. (\ref{2})
	\begin{equation}
		\begin{split}
			&\mathcal{R}^{pc}_{ic,q}(\mathbf{x}_i, 0) =  u_q^M - g_q  + \epsilon u_q^{pc},\\
			&\mathcal{R}^{pc}_{bc,q}(\mathbf{x}_i, t_i) =  u_q^M - h_q  + \epsilon u_q^{pc}.
		\end{split}
	\end{equation}
	Their associated Jacobian entries formulated concerning $\lambda_{j,q}$ read as follows
	\begin{equation}
		\begin{split}
			&(\mathcal{J}^{pc}_{ic,q})_{i, j} =\left. \frac{\partial \mathcal{R}_{ic,q}^{pc}}{\partial \lambda_{j,q}}\right|_{(\mathbf{x}_i, 0)} = \epsilon \, \tau_{j,q}^u,\\
			&(\mathcal{J}^{pc}_{bc,q})_{i, j} =\left. \frac{\partial \mathcal{R}_{bc,q}^{pc}}{\partial \lambda_{j,q}}\right|_{(\mathbf{x}_i, t_i)} = \epsilon \, \tau_{j,q}^u.
		\end{split}
	\end{equation}
	
	For points allocated on the moving interface $(\mathbf{x}_i, t_i) \in N_{bs}^{pc}$, the Taylor expansion of the isothermal boundary residual $\mathcal{R}_{bs,q}^{pc}$ up to $\mathcal{O}(\epsilon^2)$ yields
	\begin{equation}
		\begin{split}
			\mathcal{R}_{bs,q}^{pc} &\approx  \epsilon^{0} u_q^M(R^H)   \\
			&+ \epsilon^{1}  \left( R^{pc} \cdot \nabla u_q^M(R^H) + u_q^{pc}(R^H) \right)   \\
			&+ \epsilon^2 \left( \frac{1}{2} (R^{pc})^T H_x[u_q^M](R^H) R^{pc} + R^{pc} \cdot \nabla u_q^{pc}(R^H) \right),
		\end{split}
	\end{equation}
	where $H_x[u_q^M]$ denotes the spatial Hessian matrix. The corresponding Jacobian block matrices $\mathcal{J}^{pc}_{bs,uq}$ and $\mathcal{J}^{pc}_{bs,R}$ evaluate to
	\begin{equation}
		\begin{split}
			(\mathcal{J}^{pc}_{bs,uq})_{i, j} = \left.\frac{\partial \mathcal{R}_{bs,q}^{pc}}{\partial \lambda_{j,q}} \right|_{(\mathbf{x}_i, t_i)} =& \, \epsilon \tau_{j,q}^u + \epsilon^2 \left( R^{pc} \cdot \nabla \tau_{j,q}^u \right), \\
			(\mathcal{J}^{pc}_{bs,R})_{i, j} = \left.\frac{\partial \mathcal{R}_{bs,q}^{pc}}{\partial \theta_j}\right|_{(\mathbf{x}_i, t_i)} =& \, \epsilon \left( \tau_j^R \cdot \nabla u_q^M \right) + \epsilon^2 \left( \tau_j^R \cdot \nabla u_q^{pc} + (\tau_j^R)^T H_x[u_q^M] R^{pc} \right).
		\end{split}
	\end{equation}
	
	For points within the initial free boundary set $N_{fb}^{pc}$, the perturbation-correction residual simplifies to
	\begin{equation}
		\mathcal{R}^{pc}_{fb}(\mathbf{x}_i, 0) =  R^H - R_{0}  + \epsilon R^{pc}.
	\end{equation}
	Its derivative with respect to $\theta_j$ is given exclusively by
	\begin{equation}
		(\mathcal{J}^{pc}_{fb,R})_{i, j} =\left.\frac{\partial \mathcal{R}_{fb}^{pc}}{\partial \theta_{j}}\right|_{(\mathbf{x}_i, 0)} = \epsilon \tau_{j}^R .
	\end{equation} 
	
	Considering the normal vector $\mathbf{n}$ is projected at the intermediate background interface $R^H$, we similarly expand the normal derivative components for the Stefan condition $\mathcal{R}_{sc}^{pc}$
	\begin{equation}
		\begin{split}
			&\mathcal{R}_{sc}^{pc}\\
			&\approx  \epsilon^{0}\left(\sum_{q=1}^2 (-1)^{q+1}k_q \partial_{\mathbf{n}} u_q^M(R^H) - h(R^H)\right) \\
			&+ \epsilon  \left( \sum_{q=1}^2 (-1)^{q+1}k_q \left( \partial_{\mathbf{n}} u_q^{pc}(R^H) + R^{pc} \cdot \nabla(\partial_{\mathbf{n}} u_q^M) \right) - R^{pc} \cdot \nabla h \right)  \\
			&+ \epsilon^2  \left( \sum_{q=1}^2 (-1)^{q+1}k_q \left( R^{pc} \cdot \nabla(\partial_{\mathbf{n}} u_q^{pc}) + \frac{1}{2} (R^{pc})^T H_x[\partial_{\mathbf{n}} u_q^M] R^{pc} \right)\right.\\
			&\left. - \frac{1}{2} (R^{pc})^T H_x[h] R^{pc} \right) .
		\end{split}
	\end{equation}
	
	The complementary Jacobian mappings for $\mathcal{R}_{sc}^{pc}$ follow systematically as
	\begin{equation}
		\begin{split}
			(\mathcal{J}^{pc}_{sc,uq})_{i, j} = \left.\frac{\partial \mathcal{R}_{sc}^{pc}}{\partial \lambda_{j,q}}\right|_{(\mathbf{x}_i, t_i)} =&  \epsilon (-1)^{q+1}k_q \partial_{\mathbf{n}} \tau_{j,q}^u + \epsilon^2 (-1)^{q+1}k_q \left( R^{pc} \cdot \nabla (\partial_{\mathbf{n}} \tau_{j,q}^u) \right), \\
			(\mathcal{J}^{pc}_{sc,R})_{i, j} = \left.\frac{\partial \mathcal{R}_{sc}^{pc}}{\partial \theta_j}\right|_{(\mathbf{x}_i, t_i)} =& \, \epsilon \left( \sum_{q=1}^2 (-1)^{q+1}k_q \left( \tau_j^R \cdot \nabla(\partial_{\mathbf{n}} u_q^M) \right) - \tau_j^R \cdot \nabla h \right) \\
			+& \epsilon^2 \left( \sum_{q=1}^2 (-1)^{q+1}k_q \left( \tau_j^R \cdot \nabla(\partial_{\mathbf{n}} u_q^{pc})\right.\right.\\
			+& \left.\left.(\tau_j^R)^T H_x[\partial_{\mathbf{n}} u_q^M] R^{pc} \right)- (\tau_j^R)^T H_x[h] R^{pc} \right).
		\end{split}
	\end{equation}
	
	Gathering the components, the global Gauss--Newton Jacobian forms the block-structured matrix
	\begin{equation}
		\mathcal{J}_p = \frac{\partial \mathcal{T}_p}{\partial \boldsymbol{\xi}} = \frac{1}{\epsilon}
		\begin{pmatrix}
			\mathcal{J}^{pc}_{e} & 0 \\
			\mathcal{J}^{pc}_{ic} & 0 \\
			\mathcal{J}^{pc}_{bc} & 0 \\
			\mathcal{J}^{pc}_{bs,u} &\mathcal{J}^{pc}_{bs,R} \\
			0 & \mathcal{J}^{pc}_{fb,R} \\
			\mathcal{J}^{pc}_{sc,u} & \mathcal{J}^{pc}_{sc,R}
		\end{pmatrix},
	\end{equation}
	with the submatrices evaluated over their corresponding subsets discretizing the computational domain. 
	
	With the explicit construction of the perturbation residual vector $\mathcal{T}_p$ and the global Jacobian $\mathcal{J}_p$, the associated least-squares subproblem Eq. (\ref{eq:gamma_opt}) can be rigorously optimized via the Gauss--Newton method, mirroring the routine configured for the initial stage. Integrating both mechanisms, we synthesize the entire two-step workflow in the succeeding algorithm. The perturbation correction not only bolsters numerical accuracy but, as elucidated in Section 4, ensures a strictly convex quadratic formulation sequentially for the correction sequence, engendering predictable parameter updates and well-conditioned convergence.
	\begin{remark}
		In both the ELM and PCELM frameworks for solving the Stefan problem, a dynamic resampling strategy is employed on the free boundary following each iteration to track its morphological changes.
	\end{remark}
	\section{Convexity analysis of PCELM method}
	\par \par Following the perturbation expansion $u_q^g = u_q^M + \epsilon u_q^{pc}$ and $R^{g} = R^{H} + \epsilon R^{pc}$, we define the incremental state parameters as $\mathbf{f} = [u^{pc}_q, R^{pc}]^\top$. The high-order boundary residual $\mathcal{R}(\mathbf{f})$ at the baseline interface $R^{H}$ can be truncated up to $O(\epsilon^2)$ as follows:
	\begin{equation}
		\mathcal{R}(\mathbf{f}) = \epsilon^{0} F_0 + \epsilon^{1} F_1 + \epsilon^{2} F_2,
	\end{equation}
	where $F_0 = u_q^{M}(R^{H})$, $F_1 = u_q^{pc} + u^{M}_{q,x}R^{pc}$, and the second-order geometric correction is exactly given by:
	\begin{equation}
		F_2 = u_{q,x}^{pc}R^{pc} + \frac{1}{2}u_{q,xx}^{M}(R^{pc})^2.
	\end{equation}
	The objective of the perturbation correction stage, evaluated over the boundary domain (or discretized points), is to minimize the corresponding $L^2$-norm of the residual:
	\begin{equation}
		J(\mathbf{f}) = \frac{1}{2} \| \mathcal{R}(\mathbf{f}) \|^2_{l^2},
	\end{equation}
	where the residual vector $\mathcal{R}$ is evaluated over a set of points.
	\par To rigorously analyze the local convexity, we examine the Hessian $\mathbf{H}_J$ with respect to the continuous domain variables. Applying the generalized Newton formulation, the total Hessian decomposes into:
	\begin{equation}
		\mathbf{H}_J = (\nabla_\mathbf{f} \mathcal{R})(\nabla_\mathbf{f} \mathcal{R})^\top + \mathcal{R} \nabla^2_\mathbf{f} \mathcal{R}\\
		:=\mathbf{H}_{GN}+\mathbf{H}_{NL}.
	\end{equation}
	The local gradient vector evaluating at a generic boundary location takes the form:
	\begin{equation}
		\nabla_\mathbf{f} \mathcal{R} = \begin{pmatrix} \frac{\partial \mathcal{R}}{\partial u_{q}^{pc}} \\ \frac{\partial \mathcal{R}}{\partial R^{pc}} \end{pmatrix} \approx \begin{pmatrix} \epsilon \\ \epsilon u_{q,x}^{M} + \epsilon^2 (u_{q,x}^{pc} + u_{q,xx}^{M} R^{pc}) \end{pmatrix}.
	\end{equation}
	
	Here, $\mathbf{H}_{GN}$ is structurally positive semi-definite. For a multi-point integral set or multi-phase coupling, the aggregated $\mathbf{H}_{GN}$ spans the parameter subspace and accumulates to a strictly positive definite condition, yielding a dominant minimal positive eigenvalue $\lambda_{min}(\mathbf{H}_{GN}) > 0$. Meanwhile, the second-order derivative matrix $\nabla^2_\mathbf{f} \mathcal{R}$ introduces indefinite geometry-gradient interactions at $O(\epsilon^2)$:
	\begin{equation}
		\nabla^2_\mathbf{f} \mathcal{R} \approx \epsilon^2 \begin{pmatrix} 0 & \frac{\partial (u_{q,x}^{pc})}{\partial (u_q^{pc})} \\ \frac{\partial (u_{q,x}^{pc})}{\partial (u_q^{pc})} & u_{q,xx}^{M} \end{pmatrix}.
	\end{equation}
	
	\begin{theorem} 
		The functional $J(\mathbf{f})$ constitutes a locally convex optimization basin, provided a sufficiently small $\epsilon$ and $\mathcal{R} \to 0$.
	\end{theorem}
	\begin{proof}
		For the sub-problem to be strictly convex, we require $\mathbf{H}_J \succ 0$. Applying Weyl's inequality for Hermitian matrices yields
		\begin{equation}
			\lambda_{min}(\mathbf{H}_J) \ge \lambda_{min}(\mathbf{H}_{GN}) - \epsilon^2 \|\mathcal{H}_{NL}\|_2.
		\end{equation}
		In the vicinity of the solution, the residual approaches zero ($\mathcal{R} \to 0$). Consequently, the magnitude of the indefinite nonlinear variations $\epsilon^2 \mathcal{R} \nabla^2_\mathbf{f} \mathcal{R}$ asymptotically vanishes. 
		Provided that the spatial summation over proper boundary points grants $\lambda_{min}(\mathbf{H}_{GN}) > \mathcal{O}(\epsilon^2 |\mathcal{R}|)$, the positive eigenvalues of the Gauss-Newton matrix $\mathbf{H}_{GN}$ strictly dominate the scaled indefinite curvature terms. Specifically, embedding the $O(\epsilon^2)$ compensation acts as an asymptotic geometric damper, preventing singular curvature explosions. Thus, the optimization unconditionally maintains positive definiteness, ensuring a stable, locally convex basin.
	\end{proof}
	\begin{remark}
		While the perturbation correction stage preserves the topological structure of the original free boundary, it exhibits significantly reduced nonlinearity compared to the primary Stefan problem. Specifically, once the initialization results are established, the optimization landscape for the perturbation correction is dictated by a truncated local quadratic model. This formulation ensures that the computational complexity remains decoupled from the intrinsic nonlinearities between the phase field and the free boundary. Consequently, the correction step maintains numerical stability and efficiency, facilitating the rapid convergence observed in the empirical evaluations in Section 5.
	\end{remark}

	\section{Numerical examples}
	\label{sec5}
	In this section, a series of numerical experiments is conducted to evaluate the efficiency and robustness of the PCELM framework for solving Stefan problems. Table \ref{tab:3} presents the initial stage hyperparameters for the following five numerical examples. During the perturbation correction stage, these hyperparameter values are doubled to solve the Stefan problems.
	\par To quantify the numerical approximation, we introduce the following norms. Denote $u_{pre}(\boldsymbol{x}, t)$ and $u_{exa}(\boldsymbol{x}, t)$ as predicted solutions and exact solutions, respectively. $(\boldsymbol{x}_{i},t_{i})^{N_{K}}_{i=1}$ is denoted as the testing points on the space-time domain. To quantify the precision of the solutions, we use the relative $L_{2}$ norms uniformly, which follows:
	\begin{equation}
		||\mathrm{e}_{u}||_{r,2}=\sqrt{\frac{\sum_{i=1}^{N_{K}}\left|u_{exact}(\boldsymbol{x}_{i}, t_{i})-u_{pred}(\boldsymbol{x}_{i},t_{i}) \right|^2}{\sum_{i=1}^{N_{K}}|u_{exact}(\boldsymbol{x}_{i},t_{i})|^2}}.
	\end{equation}
	Similarly, we denote $(t_{j})_{j=1}^{N_{P}}$ as the testing points in the time domain. $R_{pre}(t)$ and $R_{exa}(t)$ are denoted as the predicted free boundary and the exact free boundary separately. Thus, the relative $L_{2}$ norms for the evolving interface can be defined as:
	\begin{equation}
		||\mathrm{e}_{R}||_{r,2}=\sqrt{\frac{\sum_{j=1}^{N_{P}}\left|R_{exact}(t_{j})-R_{pred}(t_{j}) \right|^2}{\sum_{j=1}^{N_{P}}|R_{exact}(t_{j})|^2}}.
	\end{equation}
	Moreover, to measure the numerical approximation in the Frank-sphere, the $L_{\infty}$ norm can also be defined as
	\begin{equation}
		||\mathrm{e}_{u}||_{\infty}=\max_{k\in N_{Q}}\left|u_{exact}(\boldsymbol{x}_{k},t_{k})-u_{pred}(\boldsymbol{x}_{k},t_{k}) \right|,
	\end{equation}
	where $(\boldsymbol{x}_{k},t_{k})^{N_{Q}}_{k=1}$ is denoted as the testing points on the computational domain.
	
	\begin{table}[t]
		\centering
		\caption{Network hyperparameters for the five numerical experiments, including the number of neurons and the uniform initialization ranges for weights and biases.}\label{tab:3}
		\begin{tabular}{l c c c c c }
			\hline
			\multirow{2}{*}{Case} & \multirow{2}{*}{Neurons}& \multicolumn{2}{c}{Stage 1 $u$} & \multicolumn{2}{c}{Stage 1 $R$}  \\
			\cline{3-6} 
			& & Weight & Bias & Weight & Bias \\
			\hline
			Example 5.1 & 500 & $[-2, 2]$ & $[-2, 2]$ & $[-2, 2]$ & $[-2, 2]$ \\
			Example 5.2 & 700& $[-2, 2],[-3,3]$ & $[-\pi, \pi]$ & $[-1, 1]$ & $[-\pi, \pi]$ \\
			Example 5.3 & 1000 & $[-2, 2]$ & $[-\pi, \pi]$ & $[-2, 2]$ & $[-\pi, \pi]$ \\
			Example 5.4.1 & 1500 & $[-2\pi, 2\pi]$ & $[-\pi, \pi]$ & $[-1, 1]$ & $[-2, 2]$ \\
			Example 5.4.2 & 3000 & $[-2\pi, 2\pi]$ & $[-\pi, \pi]$ & $[-1, 1]$ & $[-2, 2]$ \\
			\hline
		\end{tabular}
	\end{table}
	
	\subsection{One-dimensional one-phase Stefan problem}
	\label{subsec1}
	For the first numerical test, we examine a classical one-dimensional one-phase Stefan problem that has been considered previously in \cite{24, FP, RA}. The mathematical model, along with its initial and boundary conditions, can be expressed as
	\begin{equation}
		\frac{\partial u}{\partial t} - \frac{\partial^2 u}{\partial x^2} = 0,\ 0 \leq x \leq R(t), 0 \leq t \leq 1,
		\label{eq:37}
	\end{equation}
	\begin{equation}
		u(x, 0) = -\frac{x^2}{2} + 2x - \frac{1}{2},\ 0 \leq x \leq R(0),
		\label{eq:38}
	\end{equation}
	\begin{equation}
		\frac{\partial u(0, t)}{\partial x} = 2,\ 0 \leq t \leq 1,
	\end{equation}
	The moving interface satisfies the following conditions:
	\begin{equation}
		R(0) = 2 - \sqrt{3}, u(R, t) = 0,\ 0 \leq t \leq 1, \frac{\partial u(R, t)}{\partial x} = \sqrt{3 - 2t},\ 0 \leq t \leq 1.
	\end{equation}
	The exact solution and the exact free boundary are expressed by
	\begin{equation}
		u(x, t) = -\frac{x^2}{2} + 2x - \frac{1}{2} - t,\ 0 \leq x \leq s(t),
	\end{equation}
	\begin{equation}
		R(t) = 2 - \sqrt{3 - 2t}.
	\end{equation}
	\par Figure \ref{fig:1} shows numerical results for the temperature field and the moving interface. Figure \ref{fig:sub7} and \ref{fig:sub8} show the absolute errors distribution before and after the perturbation correction which match closely. From Figure \ref{fig:sub1}, we can see the history of RMSE and $||e_{u}||_{r,2}$, which show the similar decreasing trend in the whole training.  Figure \ref{fig:sub2} shows the point-wise absolute error of $u(x,t)$ which achieves a high order of $10^{-12}$. Figure \ref{fig:sub5} shows the comparison of the numerical and exact free boundaries, where their high degree of consistency indicates the effectiveness of the proposed method. Figure \ref{fig:sub6} presents the time evolution of absolute error of $R(t)$ over the time domain $(0,1]$.  It illustrates the highly accuracy of the PCELM method for solving the Stefan problem.
	\par In the next numerical experiment, we test the PCELM method for the one-dimensional two-phase Stefan problem.

	\begin{figure}[p] 
		\centering
		\begin{subfigure}[b]{0.45\textwidth}
			\centering
			\includegraphics[width=\textwidth]{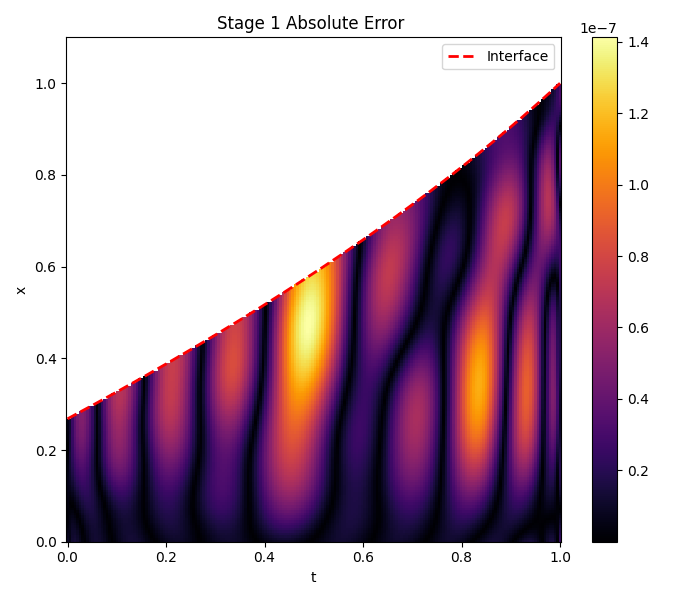} 
			\caption{}
			\label{fig:sub7}
		\end{subfigure}
		\hfill
		\begin{subfigure}[b]{0.45\textwidth}
			\centering
			\includegraphics[width=\textwidth]{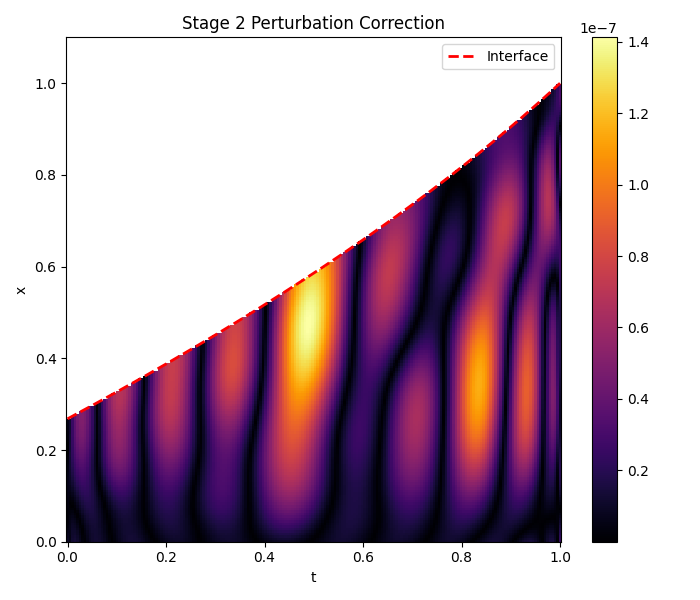} 
			\caption{}
			\label{fig:sub8}
		\end{subfigure}
		\begin{subfigure}[b]{0.45\textwidth}
			\centering
			\includegraphics[width=\textwidth]{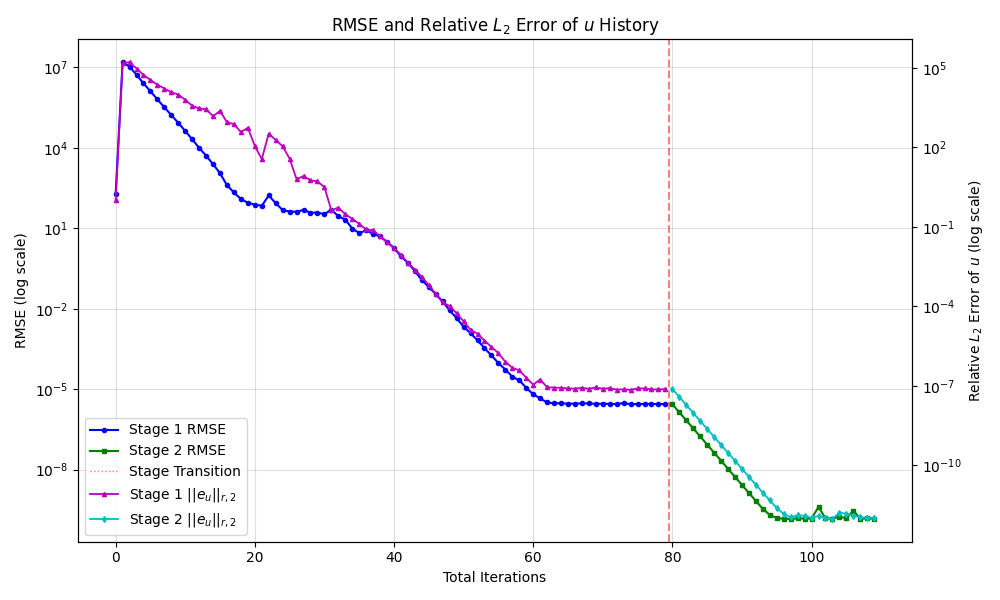}
			\caption{}
			\label{fig:sub1}
		\end{subfigure}
		\hfill
		\begin{subfigure}[b]{0.45\textwidth}
			\centering
			\includegraphics[width=\textwidth]{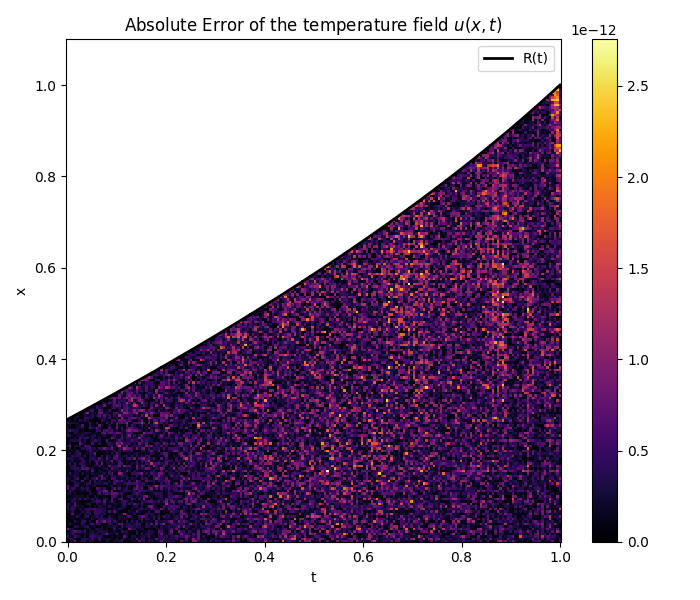}
			\caption{}
			\label{fig:sub2}
		\end{subfigure}
		
		\begin{subfigure}[b]{0.45\textwidth}
			\centering
			\includegraphics[width=\textwidth]{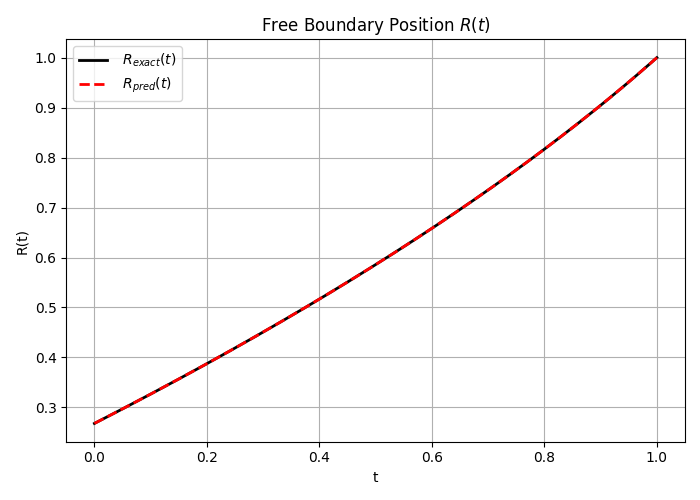} 
			\caption{}
			\label{fig:sub5}
		\end{subfigure}
		\hfill
		\begin{subfigure}[b]{0.45\textwidth}
			\centering
			\includegraphics[width=\textwidth]{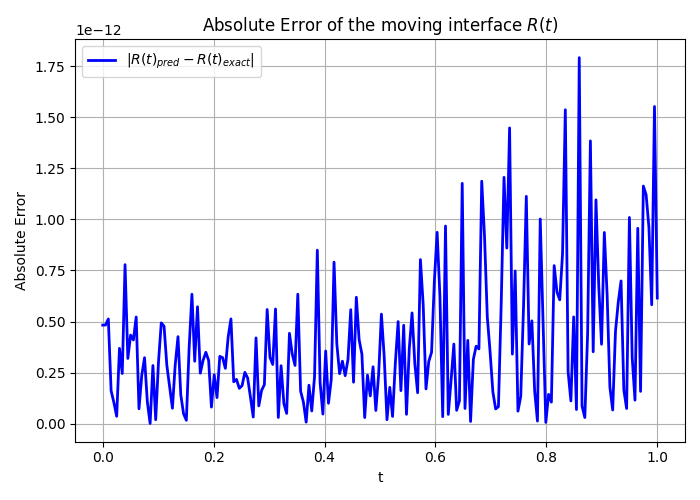} 
			\caption{}
			\label{fig:sub6}
		\end{subfigure}

		\caption{Numerical results for the one-dimensional one-phase Stefan problem: 
			(a) Absolute error distribution of the phase in Stage 1;
			(b) Visualization of the phase in the perturbation correction;
			(c) Evolution of the residual norm and Relative $L_{2}$ Error of the phase in the training; 
			(d) Point-wise absolute error of the temperature field;
			(e) Free boundary position comparison;
			(f) Evolution of the absolute error of the moving interface. }
		\label{fig:1}
	\end{figure}
	\subsection{One-dimensional two-phase Stefan problem}\label{subsec2}
	To further test efficiency of the proposed method, one-dimensional two-phase Stefan problem is chosen as the second numerical experiment \cite{24, FP, BD}. Unlike the one-phase Stefan problem, two-phase Stefan problem simulates the two phases evolve with the time-dependent free boundary. We consider the following governing equation for the two phases: 
	\begin{equation}
		\frac{\partial u_1}{\partial t} - 2\frac{\partial^2 u_1}{\partial x^2} = 0,\ 0 \leq x \leq R(t), 0 \leq t \leq 1,
	\end{equation}
	\begin{equation}
		\frac{\partial u_2}{\partial t} - \frac{\partial^2 u_2}{\partial x^2} = 0,\ R(t) \leq x \leq 2, 0 \leq t \leq 1.
	\end{equation}
	The initial and the boundary conditions are given by: 
	\begin{equation}
		u_1(x, 0) = 2e^\frac{1-2x}{4} - 2,\ 0 \leq x \leq R(0), u_2(x, 0) = e^\frac{1-2x}{2} - 1,\ R(0) \leq x \leq 2,
	\end{equation}
	\begin{equation}
		u_1(0, t) = 2e^\frac{2t+1}{4} - 2,\ 0 \leq t \leq 1, u_2(2, t) = e^\frac{2t-3}{2} - 1,\ 0 \leq t \leq 1.
	\end{equation}
	The free boundary satisfies the following conditions:
	\begin{equation}
		R(0) = \frac{1}{2}, u_1(R, t) = 0,\ 0 \leq t \leq 1, u_2(R, t) = 0,\ 0 \leq t \leq 1,
	\end{equation}
	\begin{equation}
		\frac{\text{d}R}{\text{d}t} = -2\frac{\partial u_1(R, t)}{\partial x} + \frac{\partial u_2(R, t)}{\partial x},\ 0 \leq t \leq 1.
	\end{equation}
	The exact solution is given by:
	\begin{equation}
		u_1(x, t) = 2e^\frac{2t-2x+1}{4} - 2,\ 0 \leq x \leq R(t),
	\end{equation}
	\begin{equation}
		u_2(x, t) = e^\frac{2t-2x+1}{2} - 1,\ R(t) \leq x \leq 2,
	\end{equation}
	\begin{equation}
		R(t) = t + 0.5.
	\end{equation}
	\par Figure \ref{fig:2} shows the comparisons of the exact solutions and the numerical solutions, including the temperature field and the free boundary. Figure \ref{fig:sub27} and \ref{fig:sub28} exhibit the spatiotemporal distribution of the absolute error before and after the perturbation correction, respectively. Figure \ref{fig:sub21} shows the evolution of residual norm and $||e_{R}||_{r,2}$ by the PCELM method. Figure \ref{fig:sub22} shows the absolute error distribution of the temperature field. Notably, Figure \ref{fig:sub25} depicts the numerical and exact free boundary position which matches very well. Furthermore, Figure \ref{fig:sub26} reveals that the absolute error associated with the moving interface is suppressed to an exceptionally low magnitude of $10^{-11}$. Taken together, these findings rigorously validate the capability and precision of the PCELM method in solving the Stefan problem.
	\par In the following subsection, we will apply our method for solving the two-dimensional one-phase Stefan problem.
	
	\begin{figure}[p] 
		\centering
		\begin{subfigure}[b]{0.42\textwidth}
			\centering
			\includegraphics[width=\textwidth]{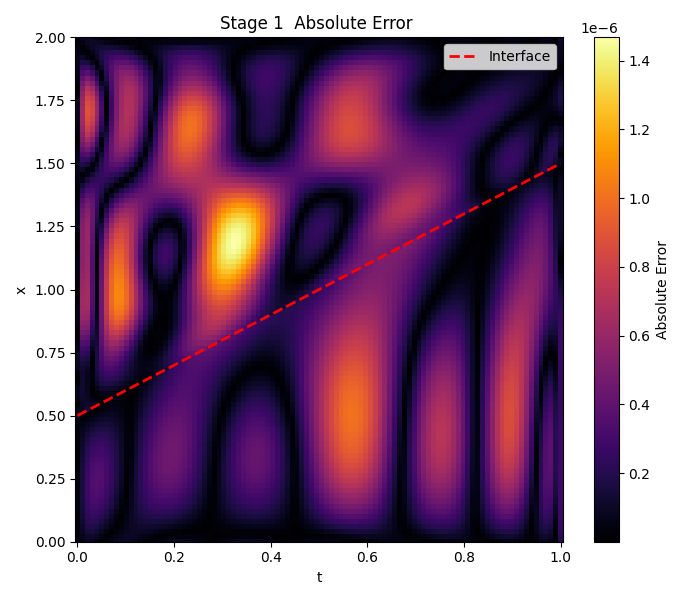} 
			\caption{}
			\label{fig:sub27}
		\end{subfigure}
		\hfill
		\begin{subfigure}[b]{0.42\textwidth}
			\centering
			\includegraphics[width=\textwidth]{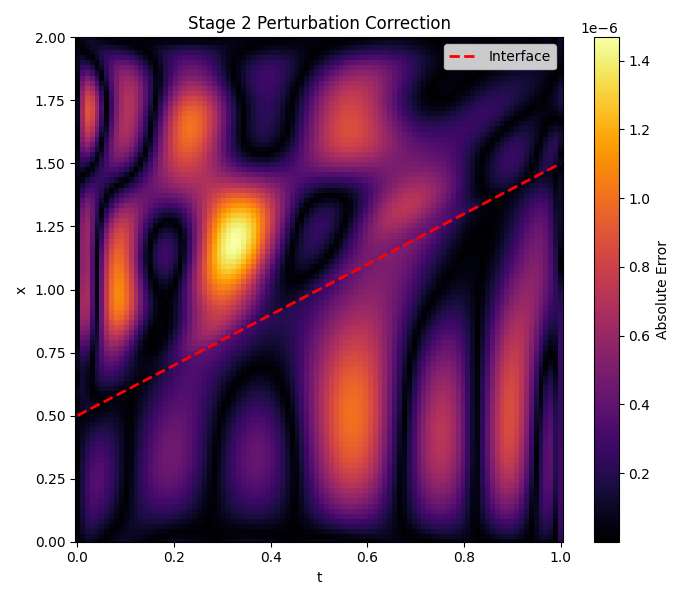} 
			\caption{}
			\label{fig:sub28}
		\end{subfigure}
		\begin{subfigure}[b]{0.42\textwidth}
			\centering
			\includegraphics[width=\textwidth]{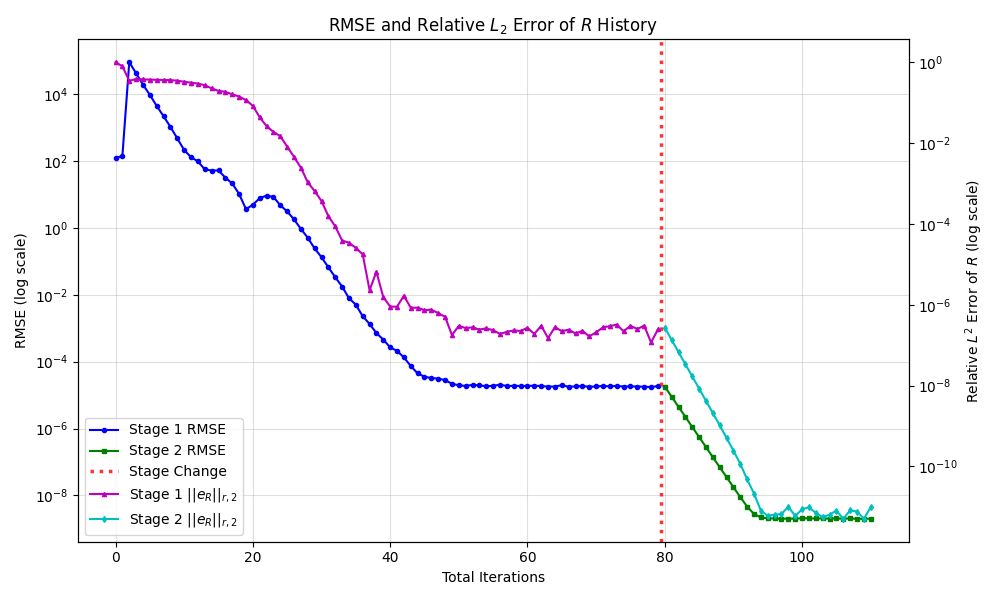}
			\caption{}
			\label{fig:sub21}
		\end{subfigure}
		\hfill
		\begin{subfigure}[b]{0.42\textwidth}
			\centering
			\includegraphics[width=\textwidth]{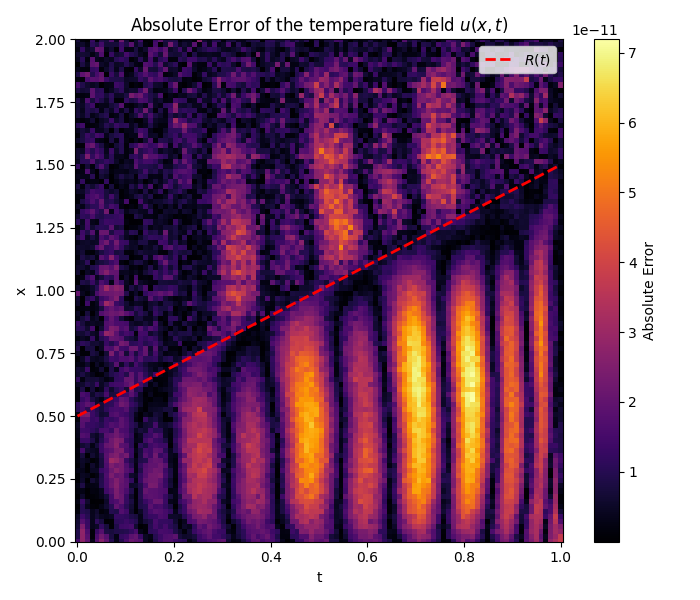}
			\caption{}
			\label{fig:sub22}
		\end{subfigure}

		\begin{subfigure}[b]{0.42\textwidth}
			\centering
			\includegraphics[width=\textwidth]{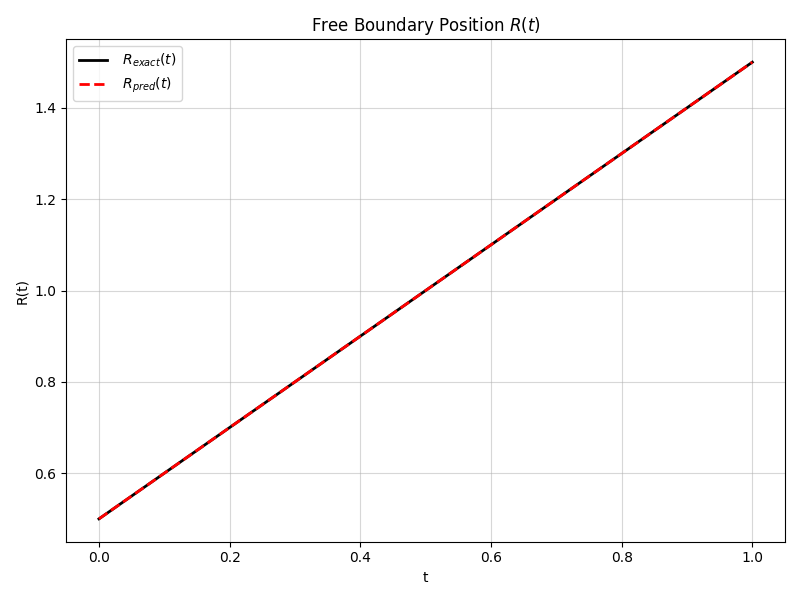} 
			\caption{}
			\label{fig:sub25}
		\end{subfigure}
		\hfill
		\begin{subfigure}[b]{0.42\textwidth}
			\centering
			\includegraphics[width=\textwidth]{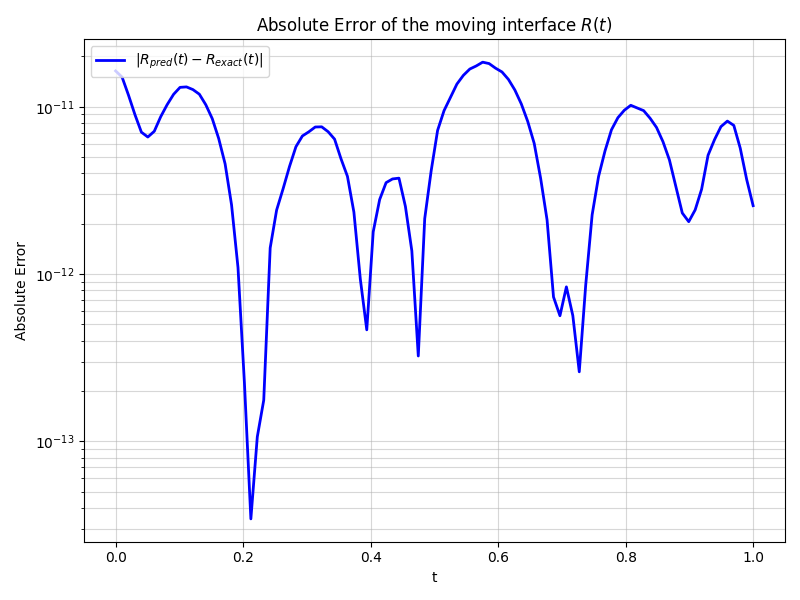} 
			\caption{}
			\label{fig:sub26}
		\end{subfigure}
		\caption{Numerical results for the one-dimensional two-phase Stefan problem: 
			(a) Absolute error distribution of the phase in Stage 1;
			(b) Visualization of the phase in the perturbation correction;
			(c) Evolution of the residual norm and Relative $L_{2}$ Error of the free boundary in the training; 
			(d) Point-wise absolute error of the temperature field; 
			(e) Free boundary position comparison;
			(f) Evolution of the absolute error of the moving interface.}
		\label{fig:2}
	\end{figure}
	
	\subsection{Two-dimensional one-phase Stefan problem}\label{subsec3}
	To evaluate the performance of the proposed approach, we next study a two-dimensional Stefan problem with a moving boundary; related formulations can also be found in \cite{24, FP, DR}. The governing equation for the temperature variable is written as
	\begin{equation}
		\frac{\partial u}{\partial t} - \frac{\partial^2 u}{\partial x^2} - \frac{\partial^2 u}{\partial y^2} = 0,\ 0 \leq x \leq R(y,t),\ 0 \leq y \leq 1,\ 0 \leq t \leq 1,
	\end{equation}
	
	\begin{equation}
		u(x,y,0) = e^{-x+\frac{1}{2}y+\frac{1}{2}} - 1,\ 0 \leq x \leq R(y,0),\ 0 \leq y \leq 1,
	\end{equation}
	
	\begin{equation}
		u(0,y,t) = e^{\frac{5}{4}t+\frac{1}{2}y+\frac{1}{2}} - 1,\ 0 \leq y \leq 1,\ 0 \leq t \leq 1,
	\end{equation}
	
	\begin{equation}
		u(x,0,t) = e^{\frac{5}{4}t-x+\frac{1}{2}} - 1,\ 0 \leq x \leq R(0,t),\ 0 \leq t \leq 1,
	\end{equation}
	
	\begin{equation}
		u(x,1,t) = e^{\frac{5}{4}t-x+1} - 1,\ 0 \leq x \leq R(1,t),\ 0 \leq t \leq 1.
	\end{equation}
	
	On the moving interface $R(y,t)$, the following conditions are imposed:
	\begin{equation}
		R(y,0) = \frac{1}{2}y + \frac{1}{2},\ 0 \leq y \leq 1,\ u(R,y,t) = 0,\ 0 \leq y \leq 1,\ 0 \leq t \leq 1,
	\end{equation}
	
	\begin{equation}
		\frac{\partial u(R,y,t)}{\partial x} - \frac{\partial u(R,y,t)}{\partial y} \frac{\partial R(y,t)}{\partial y} + \frac{\partial R(y,t)}{\partial t} = 0,\ 0 \leq y \leq 1,\ 0 \leq t \leq 1.
	\end{equation}
	
	The analytical solution of this example is given by
	\begin{equation}
		u(x,y,t) = e^{\frac{5}{4}t-x+\frac{1}{2}y+\frac{1}{2}} - 1,\ 0 \leq x \leq R(y,t),
	\end{equation}
	
	\begin{equation}
		R(y,t) = \frac{5}{4}t + \frac{1}{2}y + \frac{1}{2}.
	\end{equation}
	
	\par A collection of numerical results produced by the PCELM method is displayed in Figure \ref{fig:3}. The absolute errors of the temperature field at $t=0.2$, $t=0.5$, and $t=0.8$ are shown in Figures \ref{fig:sub35}, \ref{fig:sub36}, and \ref{fig:sub37}, respectively. In all three cases, the errors remain at the level of $10^{-11}$, and the computed free boundary agrees closely with the exact one. In addition, Figure \ref{fig:sub34} depicts the evolution of the residual norm, from which one can observe a further reduction during the correction stage. The time history of the relative $L_{2}$ error for the temperature field is presented in Figure \ref{fig:sub38}. Figure \ref{fig:sub33} displays the three-dimensional distribution of the absolute error associated with the free boundary. Figures \ref{fig:sub39} and \ref{fig:sub310} provide the absolute free-boundary error before perturbation correction and the corresponding result after correction, respectively, indicating the effectiveness of the proposed PCELM approach. These observations demonstrate that the PCELM method is effective for the numerical treatment of Stefan problems. Overall, the computed results indicate that the proposed approach attains an accuracy of approximately $10^{-11}$ for this multi-dimensional example.
	
	\par Table \ref{tab:errors} lists a quantitative comparison of the relative $L_{2}$ errors obtained by PCELM, PINN, and PIELM. It can be seen that PCELM consistently delivers the smallest errors among the three methods, improving the accuracy by approximately 4--6 orders of magnitude for both the temperature field and the free-boundary location. This comparison further verifies the robustness and computational effectiveness of PCELM in capturing the dynamics of Stefan-type problems.
	\par In the next subsection, a classic Frank-sphere problem will be considered.
	\begin{figure}[p] 
		\centering
		\begin{subfigure}[b]{0.3\textwidth}
			\centering
			\includegraphics[width=\textwidth]{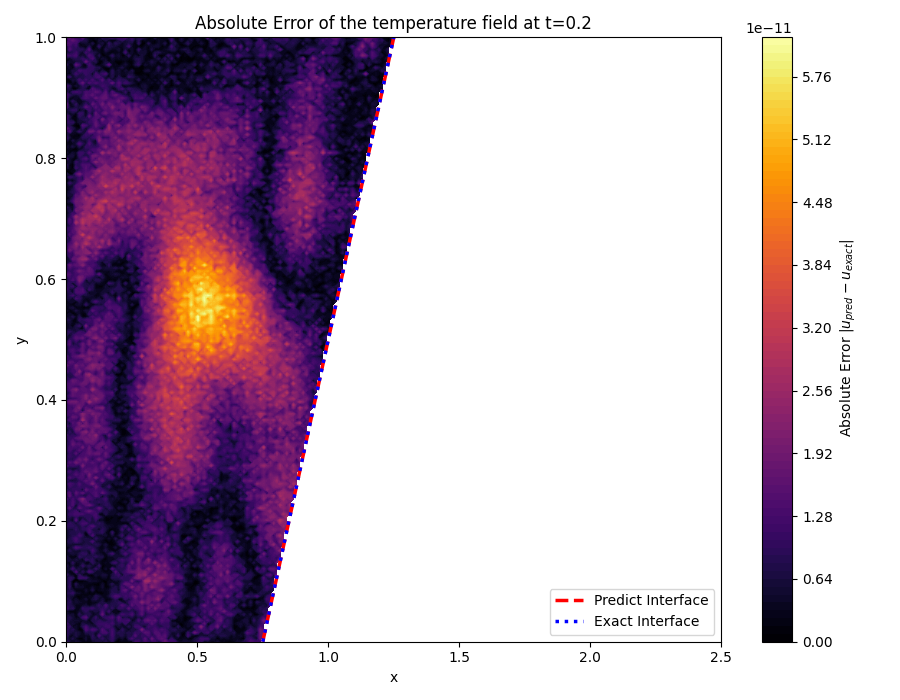} 
			\caption{}
			\label{fig:sub35}
		\end{subfigure}
		\hfill
		\begin{subfigure}[b]{0.3\textwidth}
			\centering
			\includegraphics[width=\textwidth]{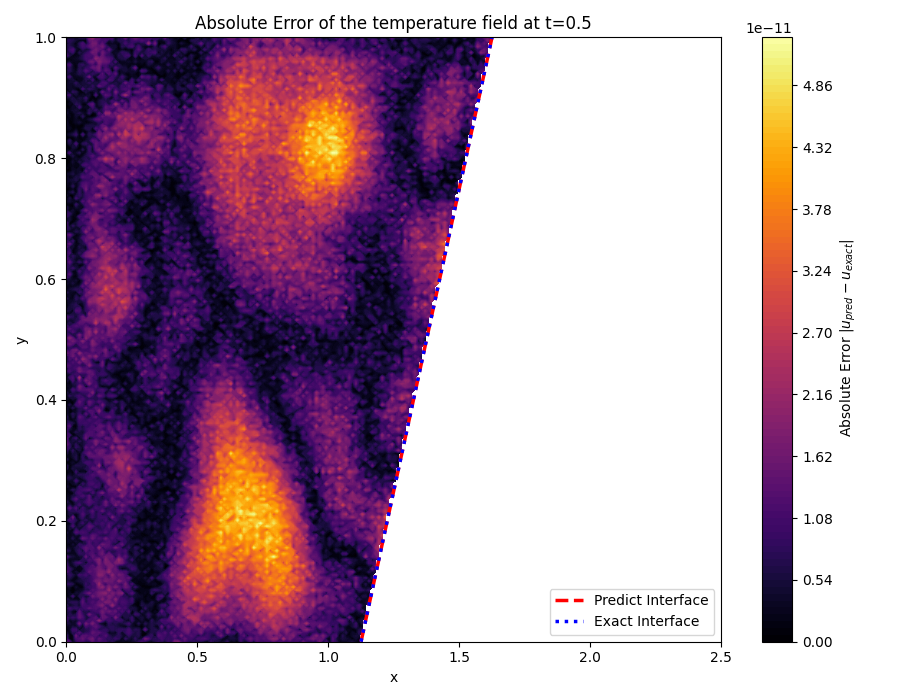} 
			\caption{}
			\label{fig:sub36}
		\end{subfigure}
		\hfill
		\begin{subfigure}[b]{0.3\textwidth}
			\centering
			\includegraphics[width=\textwidth]{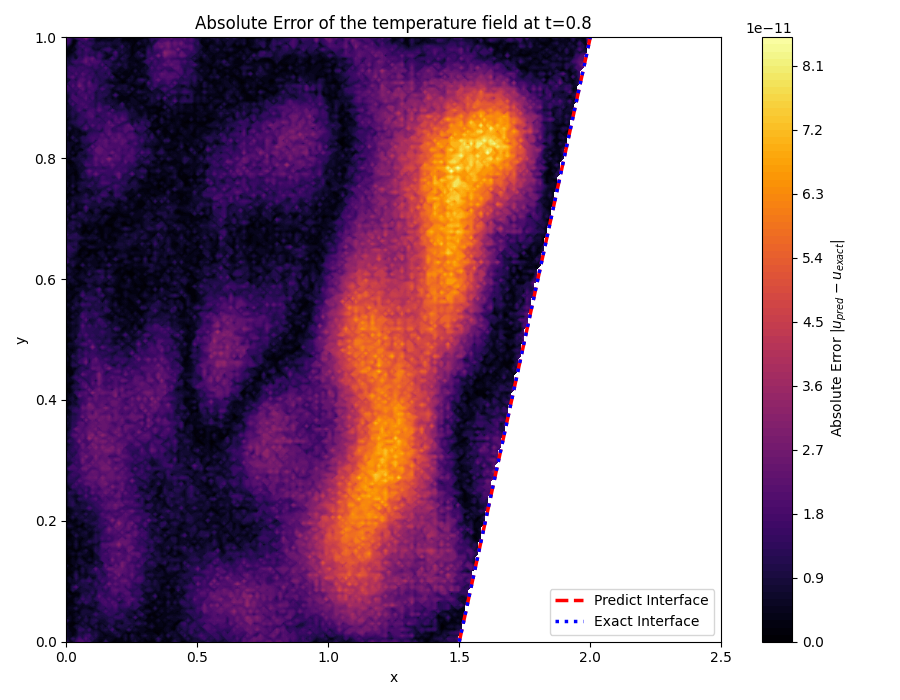} 
			\caption{}
			\label{fig:sub37}
		\end{subfigure}
		\begin{subfigure}[b]{0.3\textwidth}
			\centering
			\includegraphics[width=\textwidth]{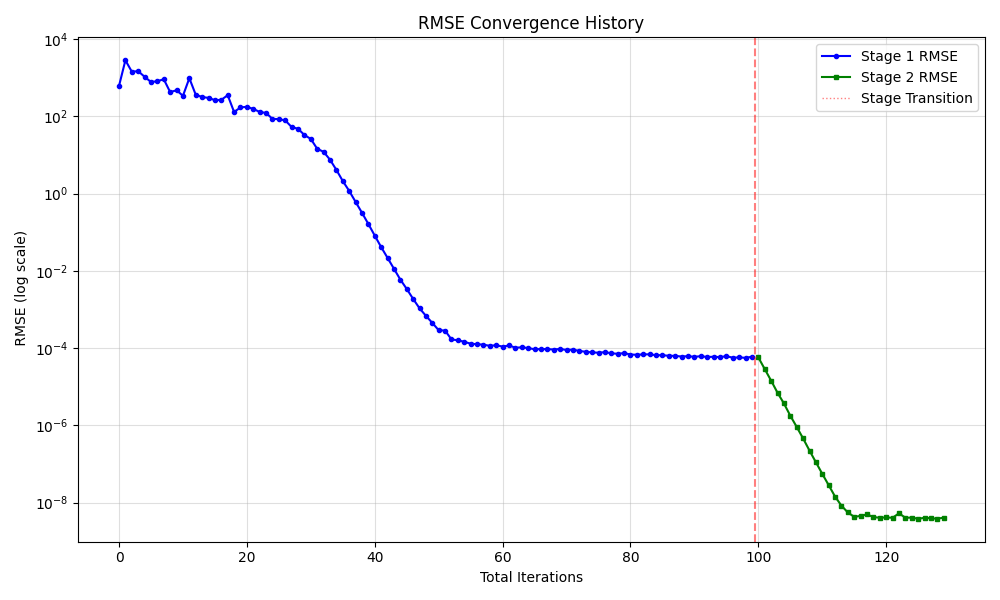}
			\caption{}
			\label{fig:sub34}
		\end{subfigure}
		\hfill
		\begin{subfigure}[b]{0.3\textwidth}
			\centering
			\includegraphics[width=\textwidth]{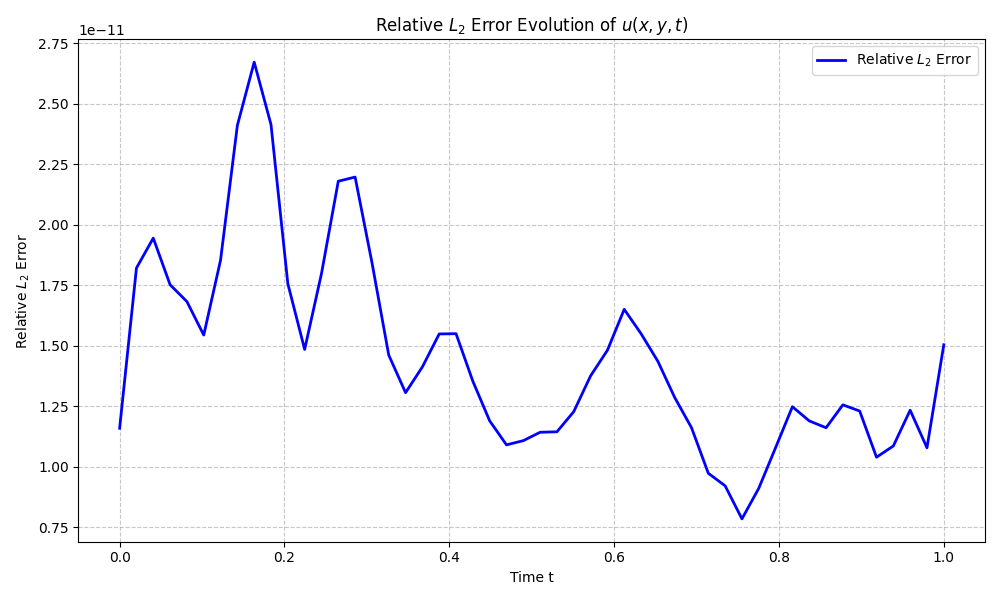}
			\caption{}
			\label{fig:sub38}
		\end{subfigure}
		\hfill
		\begin{subfigure}[b]{0.3\textwidth}
			\centering
			\includegraphics[width=\textwidth]{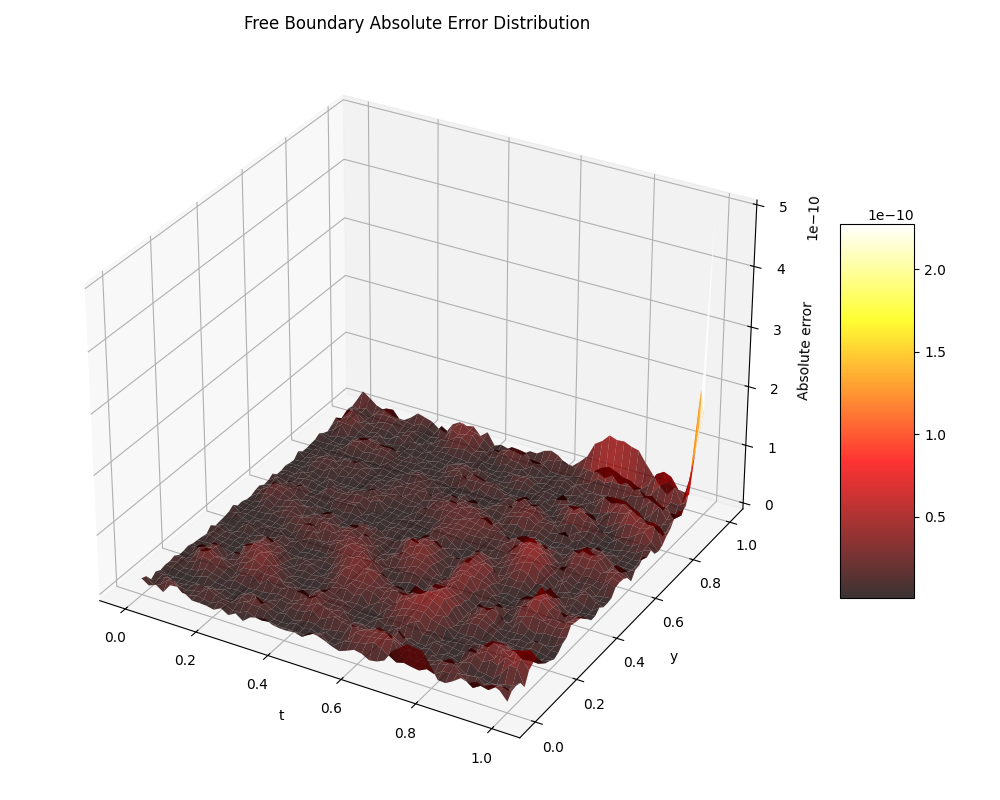}
			\caption{}
			\label{fig:sub33}
		\end{subfigure}
		\begin{subfigure}[b]{0.48\textwidth}
			\centering
			\includegraphics[width=\textwidth]{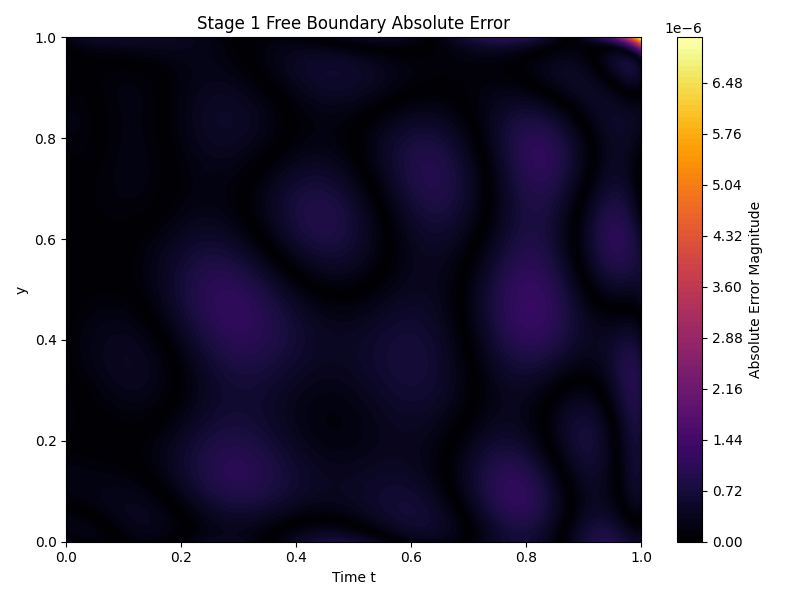} 
			\caption{}
			\label{fig:sub39}
		\end{subfigure}
		\hfill
		\begin{subfigure}[b]{0.48\textwidth}
			\centering
			\includegraphics[width=\textwidth]{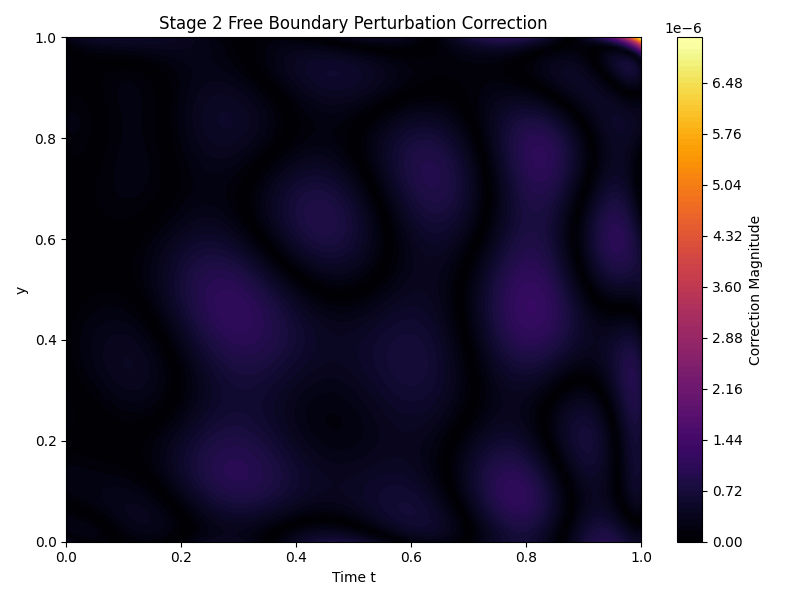} 
			\caption{}
			\label{fig:sub310}
		\end{subfigure}

		\caption{Numerical results for the two-dimensional one-phase Stefan problem: 
			(a) Point-wise absolute error of the temperature field at $t=0.2$; 
			(b) Point-wise absolute error of the temperature field at $t=0.5$;
			(c) Point-wise absolute error of the temperature field at $t=0.8$; 
			(d) RMSE iteration history in the training; 
			(e) Relative $L_{2}$ error evolution of the phase;
			(f) Free boundary absolute error distribution;
			(g) Absolute error distribution of the free boundary in Stage 1;
			(h) Visualization of the free boundary in the perturbation correction.}
		\label{fig:3}
	\end{figure}
	
	\begin{table}[t]
		\centering
		\caption{Comparison of the relative $L_{2}$ errors for different methods (PINN, PIELM and PCELM).}\label{tab:errors}
		
		\small  
		\setlength{\tabcolsep}{1.2mm} 
		
		\begin{tabular}{l c c c c c c}
			\hline
			\multirow{2}{*}{Case} & \multicolumn{2}{c}{PINN \cite{24}} & \multicolumn{2}{c}{PIELM \cite{FP}} & \multicolumn{2}{c}{PCELM} \\
			\cline{2-7}
			& $||\mathrm{e}_{u}||_{r,2}$ & $||\mathrm{e}_{R}||_{r,2}$ & $||\mathrm{e}_{u}||_{r,2}$ & $||\mathrm{e}_{R}||_{r,2}$ & $||\mathrm{e}_{u}||_{r,2}$ & $||\mathrm{e}_{R}||_{r,2}$ \\
			\hline
			Example 5.1 & 3.92$\mathrm{E}{-04}$ & 1.04$\mathrm{E}{-03}$ & 3.44$\mathrm{E}{-08}$ & 7.39$\mathrm{E}{-07}$ & 8.11$\mathrm{E}{-13}$ & 9.93$\mathrm{E}{-13}$ \\
			Example 5.2 & 7.34$\mathrm{E}{-04}$ & 3.52$\mathrm{E}{-04}$ & 1.04$\mathrm{E}{-07}$ & 4.26$\mathrm{E}{-08}$ & 3.69$\mathrm{E}{-11}$ & 9.82$\mathrm{E}{-12}$ \\
			Example 5.3 & 3.39$\mathrm{E}{-04}$ & 4.14$\mathrm{E}{-05}$ & 3.19$\mathrm{E}{-06}$ & 8.15$\mathrm{E}{-07}$ & 1.43$\mathrm{E}{-11}$ & 1.29$\mathrm{E}{-11}$ \\
			\hline
		\end{tabular}
	\end{table}
	\subsection{Frank-sphere problem}
	To validate the robustness of the PCELM method, we address the classic Frank-sphere problem in this subsection, considering both its two-dimensional and three-dimensional formulations. For comparison with other up-to-date methods \cite{LF}, we opt to monitor a specific scenario: an initial liquid-phase disk (devoid of solid phase) that expands into a supercooled liquid phase.
	\subsubsection{Two-dimensional Frank-sphere problem}
	The governing equation for the two-dimensional Frank sphere problem, together with the prescribed initial and boundary conditions, is formulated as follows:
	\begin{equation}
		\frac{\partial u(\boldsymbol{x},t)}{\partial t}= \Delta u(\boldsymbol{x},t), (\boldsymbol{x},t)\in\Omega_{F}\times(t_{0},t_{1}],
	\end{equation}
	\begin{equation}
		u(\boldsymbol{x},t_{0})=u_{exact}(\boldsymbol{x},t_{0}),\boldsymbol{x}\in\Omega_{F},u(\boldsymbol{x},t)=u_{exact}(\boldsymbol{x},t),(\boldsymbol{x},t)\in\partial\Omega_{F}\times(t_{0},t_{1}].
	\end{equation}
	The equations for the moving boundary are given by:
	\begin{equation}
		R(t_{0})=0.5, u(R(t),t)=0, \nabla u(\boldsymbol{x},t)\cdot \boldsymbol{n}+ \frac{dR(t)}{dt}=0, t\in(t_{0},t_{1}].
	\end{equation}
	The exact solution is defined as follows:
	\begin{equation*}
		u(r,t)=u(s)=
		\begin{cases}
			u_{inf}\left(1-\frac{G(s)}{G(s_{0})} \right) & \text{for } s\geq s_{0} \\
			0 & \text{for } s\leq s_{0},
		\end{cases}
	\end{equation*}
	where $r$ denotes the distance to the center of the disk, with the similarity variable defined as $s=r/\sqrt{t}$ and $G(s)=F_{1}(s^{2}/4)$. Here, $F_{1}(z)=\int_{z}^{\infty}\frac{\mathrm{e}^{-t}}{t}\mathrm{d}t$ denotes the exponential integral, and $G(s)$ serves as the similarity solution to the heat equation. The term $u_{inf}$ signifies the far-field temperature at an infinite distance from the moving interface $R(t)$. The value of $u_{inf}$ is denoted as 
	\begin{equation*}
		u_{inf}=\frac{s_{0}G(s_{0})}{2G'(s_{0})}.
	\end{equation*}
	\par For the two dimensional Frank-sphere problem, the spatial domain is defined in $\Omega_{F}=[-1,+1]^{2}\setminus B(\mathbf{0}, R(t))$, where $B(\mathbf{0}, R(t))$ denotes a circular disk of radius $R(t)$ centered at the origin. The time interval is set to $t\in (0.25,0.875]$ and $s_{0}=0.5$. In our numerical implementation, we express the solution in the liquid phase through a basis expansion of the form:
	\begin{equation*}
		u^N(r,t)=\sum_{j=1}^{N} \alpha_j \sigma(\mathbf{w}_{1,j} \cdot (r,t) + b_{1,j}),
	\end{equation*}
	where $r=\sqrt{x^{2}+y^{2}}$ denotes the radial coordinate. The discrete moving boundary is denoted by
	\begin{equation*}
		R^M(t)=\sum_{i=1}^{M}\gamma_{i}\sigma(\mathbf{w}_{2,i} \cdot (\theta,t) + b_{2,i}).
	\end{equation*}
	\par Figure \ref{fig:4} presents a series of visualizations for the two-dimensional Frank-sphere problem. The first row shows the absolute error distribution in the liquid phase at \(t=0.4\), \(0.6\), and \(0.875\), respectively, where the error remains on the order of \(10^{-10}\). The second row highlights the convergence efficiency through the residual norm evolution. Furthermore, the spatiotemporal evolution of the free boundary agrees precisely with the exact solution. Finally, we compare the absolute error distribution before and after the perturbation correction. 
	\par Table \ref{tab:2} presents a comparative analysis of the $L_\infty$ errors for the liquid phase across three different computational approaches: a level-set method with PINN (LPINN), a standard Numerical method, and the PCELM method. The numerical results indicate that the PCELM method achieves the highest precision with a significantly lower error of $4.07 \times 10^{-9}$. In contrast, the standard Numerical method and the LPINN approach yield errors of $6.98 \times 10^{-4}$ and $3.92 \times 10^{-4}$ respectively. These findings clearly demonstrate that the PCELM method provides a more accurate approximation for the liquid phase compared to the other two techniques.
	\begin{figure}[p] 
		\centering
		\begin{subfigure}[b]{0.3\textwidth}
			\centering
			\includegraphics[width=\textwidth]{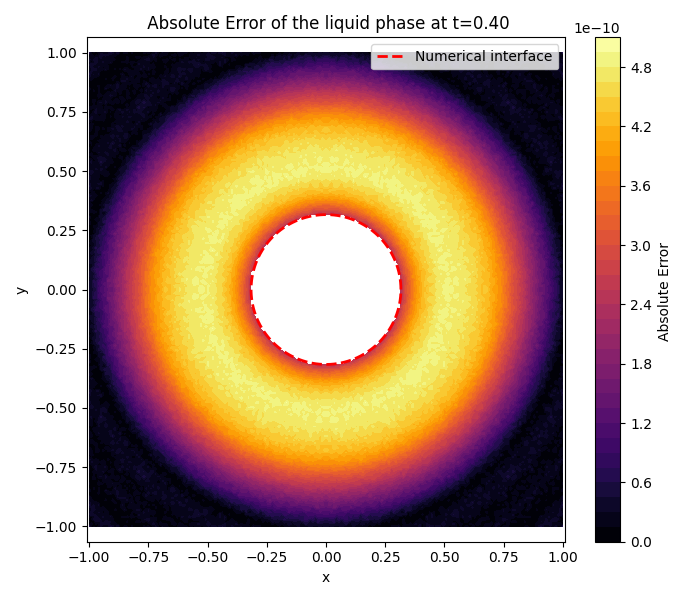} 
			\caption{}
			\label{fig:sub42}
		\end{subfigure}
		\hfill
		\begin{subfigure}[b]{0.3\textwidth}
			\centering
			\includegraphics[width=\textwidth]{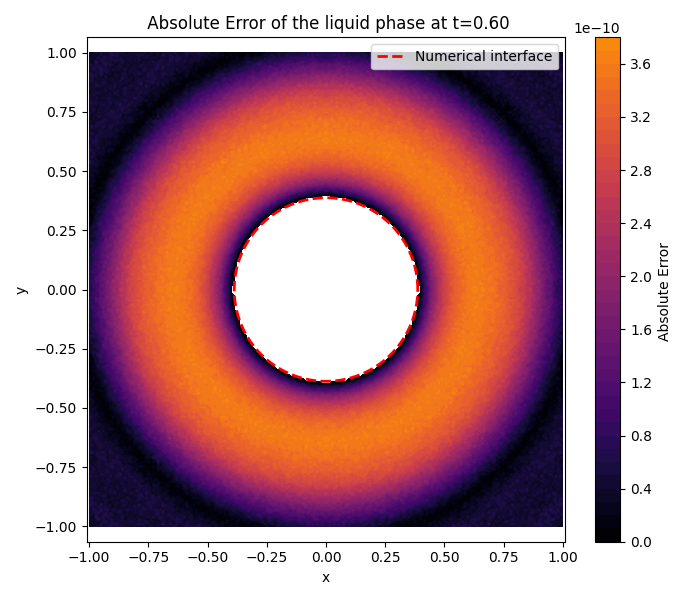} 
			\caption{}
			\label{fig:sub43}
		\end{subfigure}
		\hfill
		\begin{subfigure}[b]{0.3\textwidth}
			\centering
			\includegraphics[width=\textwidth]{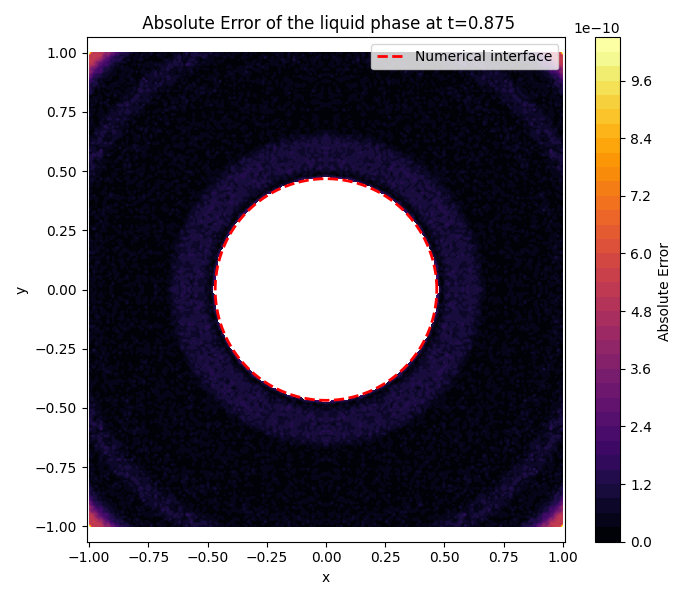} 
			\caption{}
			\label{fig:sub44}
		\end{subfigure}
		\begin{subfigure}[b]{0.48\textwidth}
			\centering
			\includegraphics[width=\textwidth]{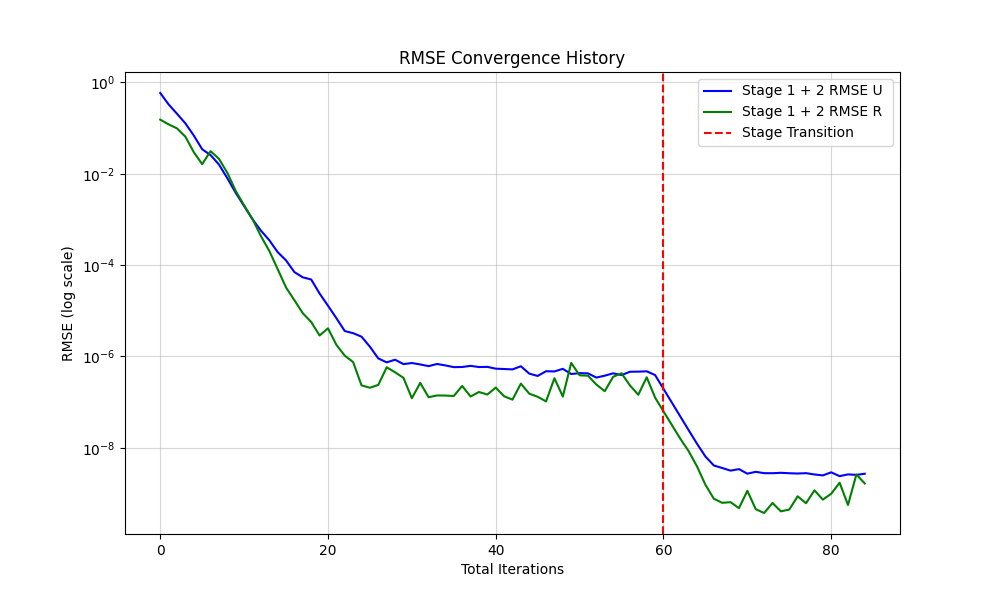}
			\caption{}
			\label{fig:sub41}
		\end{subfigure}
		\hfill
		\begin{subfigure}[b]{0.48\textwidth}
			\centering
			\includegraphics[width=\textwidth]{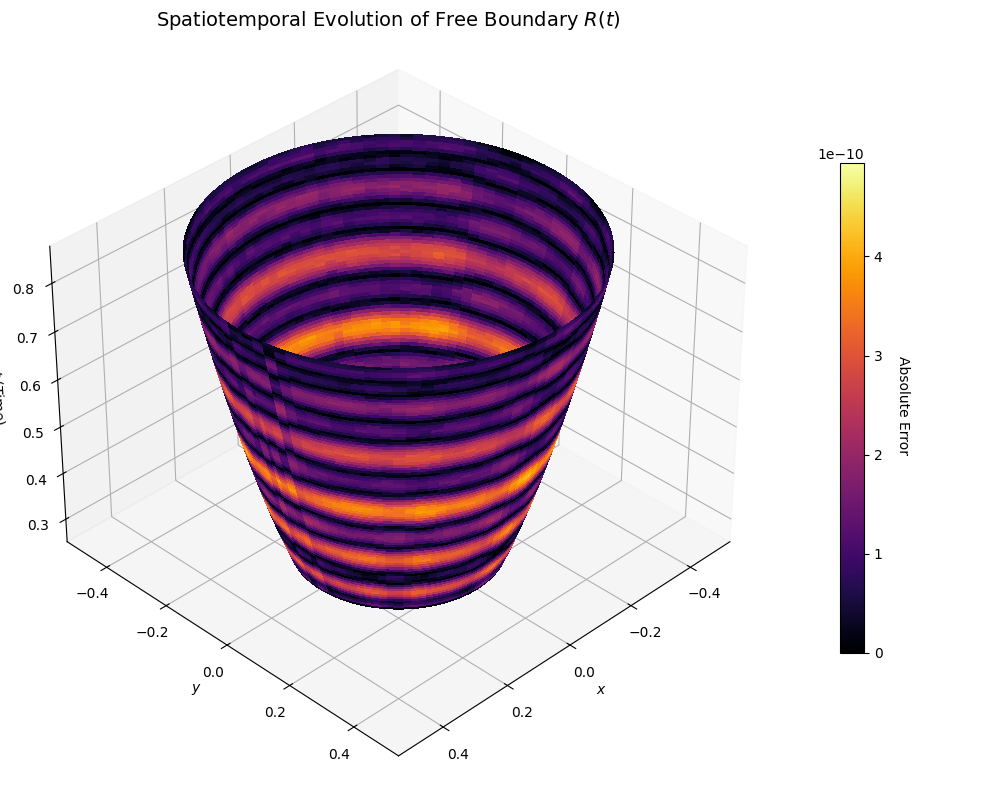}
			\caption{}
			\label{fig:sub47}
		\end{subfigure}
		\begin{subfigure}[b]{0.48\textwidth}
			\centering
			\includegraphics[width=\textwidth]{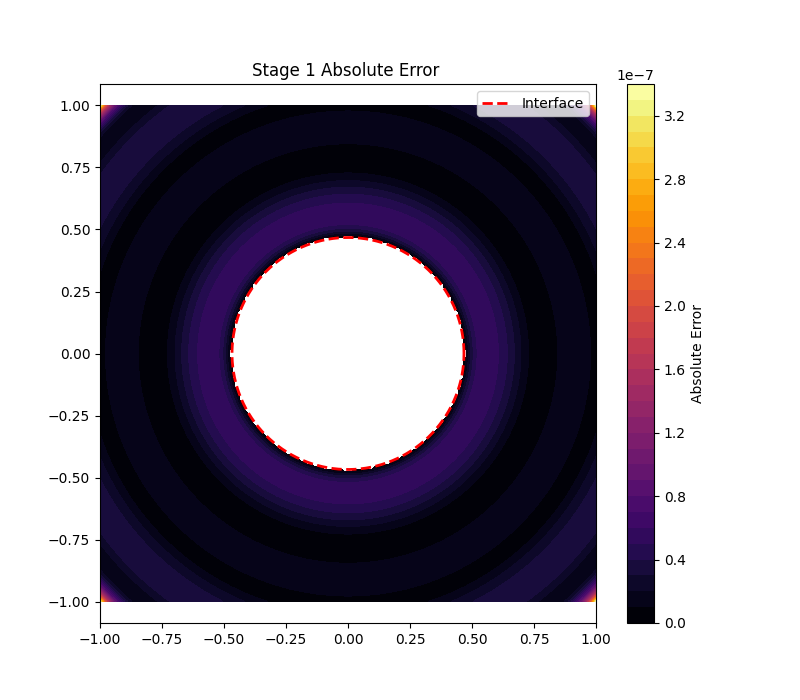} 
			\caption{}
			\label{fig:sub45}
		\end{subfigure}
		\hfill
		\begin{subfigure}[b]{0.48\textwidth}
			\centering
			\includegraphics[width=\textwidth]{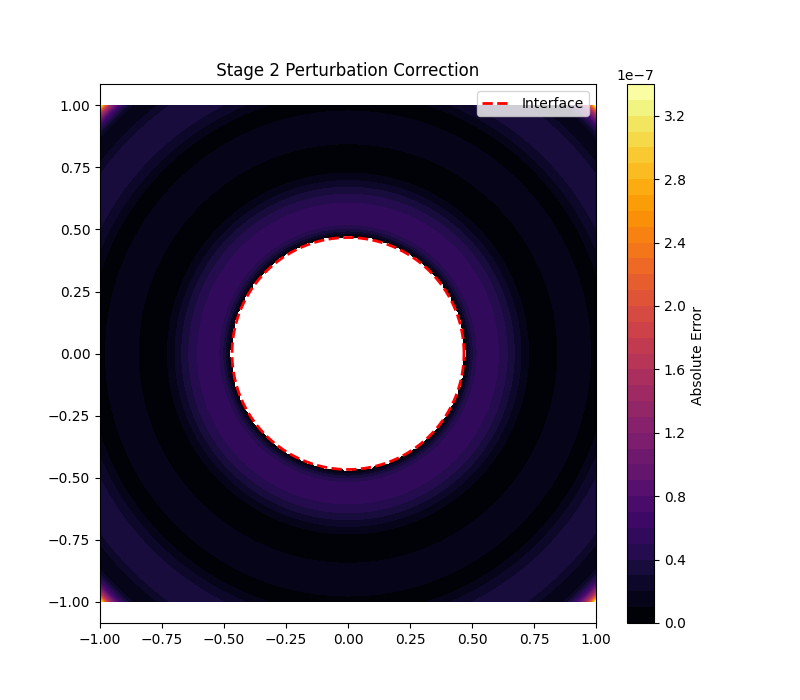} 
			\caption{}
			\label{fig:sub46}
		\end{subfigure}

		\caption{Numerical results for the two-dimensional Frank-sphere problem: 
			(a) Point-wise absolute error of the temperature field at $t=0.2$; 
			(b) Point-wise absolute error of the temperature field at $t=0.4$;
			(c) Point-wise absolute error of the temperature field at $t=0.875$; 
			(d) RMSE iteration history in the training; 
			(e) Spatiotemporal evolution of the free boundary;
			(f) Visualization of the liquid phase in Stage 1;
			(g) Visualization of the liquid phase in the perturbation correction.}
		\label{fig:4}
	\end{figure}
	
	\begin{table}[t]
		\caption{Comparison of $L_{\infty}$ errors for different methods (LPINN \cite{LF}, Numerical and PCELM).}\label{tab:2}
		\centering
		\begin{tabular}{l c }
			\hline
			\multirow{1}{*}{Method} & \multicolumn{1}{c}{$||\mathrm{e}_{u}||_{\infty}$}  \\
			\hline
			LPINN &  4.23$\mathrm{E}{-04}$ \\
			Numerical  & 6.98$\mathrm{E}{-04}$  \\
			PCELM & 4.07$\mathrm{E}{-09}$  \\
			\hline
		\end{tabular}
		
	\end{table}
	\subsubsection{Three-dimensional Frank-sphere problem}
	\par The provided computational framework solves the classical three-dimensional Frank-sphere problem, which is a well-known Stefan problem describing the continuous growth of a spherical solid phase into a surrounding undercooled liquid. Let $\Omega=[-1, +1]^3$ represent the truncated computational domain, and let $R(t)$ denote the moving liquid-solid interface at time $t$. The liquid phase domain is thus defined as $\Omega_L(t) = \Omega \setminus \{ \sqrt{x^2+y^2+z^2} \leq B(\mathbf{0}, R(t)) \}$, where $B(\mathbf{0}, R(t))$ is a ball of radius $R(t)$ centered at the origin. 
	\par With the assumption of a unit thermal diffusivity ($\kappa = 1$), the governing non-dimensional heat conduction equation and the corresponding moving boundary conditions can be mathematically formulated as follows:
	
	\begin{equation}
		\label{eq:frank_sphere_3d}
		\left\{
		\begin{aligned}
			&\frac{\partial u}{\partial t} = \Delta u, \quad &&\text{in} \ \Omega_L(t), \ t \in (T_{\text{start}}, T_{\text{end}}], \\
			&u(\mathbf{x}, t) = 0, \quad &&\text{on} \ R(t) \text{ (Isolated condition)}, \\
			&\frac{\partial R}{\partial t} + \frac{\partial u}{\partial \mathbf{n}} = 0, \quad &&\text{on} \ R(t) \text{ (Stefan Condition)}, \\
			&u(\mathbf{x}, t) = u_{\text{exact}}(\mathbf{x}, t), \quad &&\text{on} \ \partial\Omega, \ t \in (T_{\text{start}}, T_{\text{end}}], \\
			&u(\mathbf{x}, T_{\text{start}}) = u_{\text{exact}}(\mathbf{x}, T_{\text{start}}), \quad &&\text{in} \ \Omega_L(T_{\text{start}}), \\
			&R(T_{\text{start}}) = R_0 \sqrt{T_{\text{start}}}, \quad && \text{(Initial Interface Radius)},
		\end{aligned}
		\right.
	\end{equation}
	where $u(\mathbf{x}, t)$ denotes the temperature field of the liquid phase, $\mathbf{n}$ is the unit normal vector at the interface pointing towards the liquid, and $R(t)$ is the dynamic radius of the moving solid boundary. The $T_{\text{start}}$ and $T_{\text{end}}$ are $0.25$ and $1.25$, respectively.
	\par To well-pose the problem within the truncated bounded domain $\Omega$, continuous Dirichlet boundary conditions and initial conditions are enforced utilizing the analytical exact solution $u_{\text{exact}}(\mathbf{x}, t)$. The exact solution of the temperature field exhibits a self-similar nature radially, given by:
	\begin{equation*}
		\label{eq:exact_sol}
		u_{\text{exact}}(r, t) = A \left[ \frac{1}{s} \exp{\left(-\frac{s^2}{4}\right)} - \frac{\sqrt{\pi}}{2} \text{erfc}\left(\frac{s}{2}\right) \right] + T_{\infty}, \quad s = \frac{r}{\sqrt{t}},
	\end{equation*}
	where $r = \|\mathbf{x}\|_2$, $S_0 = 0.5$ is the invariant growth constant, and the coefficients $A = 0.5 S_0^3 \exp(S_0^2/4)$ and the far-field undercooling temperature $T_{\infty}$ are intrinsically determined to simultaneously satisfy the Stefan phase-change condition and the interfacial thermodynamic equilibrium.
	\par As depicted in Figure \ref{fig:5}, we present a comprehensive numerical evaluation of the proposed method applied to the three-dimensional Frank-sphere problem. Subfigures \ref{fig:sub52}, \ref{fig:sub53}, \ref{fig:sub54} illustrate the temporal evolution of the point-wise absolute error in the temperature field on the two-dimensional cross-section ($z=0$) at times $t = 0.6$, $t = 0.9$, and $t = 1.25$, respectively. The red dashed lines precisely track the moving interface, demonstrating that the numerical scheme effectively controls the approximation error within the magnitude of $10^{-7}$ to $10^{-8}$ throughout the spatial domain during the evolution process.
	\par The training dynamics are elucidated in the subfigure \ref{fig:sub51}, which plots the Root Mean Square Error (RMSE) convergence history. Subfigures \ref{fig:sub58} and \ref{fig:sub59} provide a comparative visualization of the liquid phase absolute error distributions between the Stage 1 initial approximation and the Stage 2 perturbation correction. The wave-like oscillatory error patterns indicate the high-frequency components effectively handled by the PCELM method. 
	\par Finally, subfigures \ref{fig:sub55}, \ref{fig:sub56}, \ref{fig:sub57} render the three-dimensional morphological evolution of the expanding spherical free boundary at the corresponding time steps ($t = 0.6, 0.9, 1.25$). The surface color mapping reflects the radial error between the numerical and exact interfaces, which is consistently suppressed to an exceptionally low amplitude. This visually substantiates the high fidelity of the PCELM framework in capturing dynamic, three-dimensional moving boundaries under complex phase-change scenarios.
	\par The coordinate transformation approach for the Frank-sphere problem aligns more closely with physical reality, thereby enabling the attainment of the current level of numerical accuracy. In contrast, radial-direction perturbation correction yields only a marginal improvement no greater than 4 to 6 orders of magnitude in the accuracy of the second-stage optimization, relative to numerical experiments \ref{subsec1}-\ref{subsec3}.
	
	\begin{figure}[p] 
		\centering
		\begin{subfigure}[b]{0.3\textwidth}
			\centering
			\includegraphics[width=\textwidth]{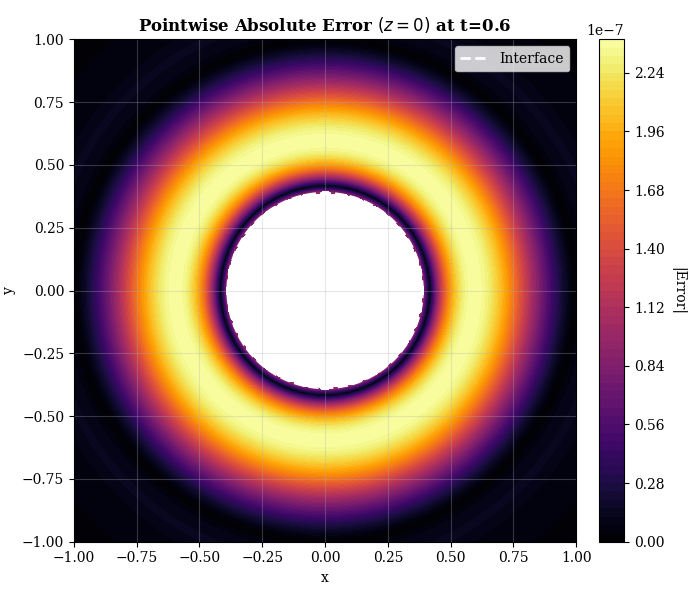} 
			\caption{}
			\label{fig:sub52}
		\end{subfigure}
		\hfill
		\begin{subfigure}[b]{0.3\textwidth}
			\centering
			\includegraphics[width=\textwidth]{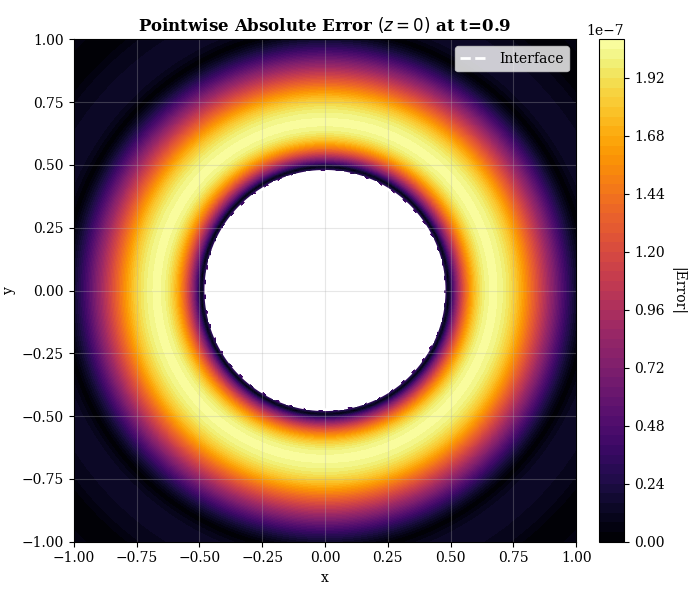} 
			\caption{}
			\label{fig:sub53}
		\end{subfigure}
		\hfill
		\begin{subfigure}[b]{0.3\textwidth}
			\centering
			\includegraphics[width=\textwidth]{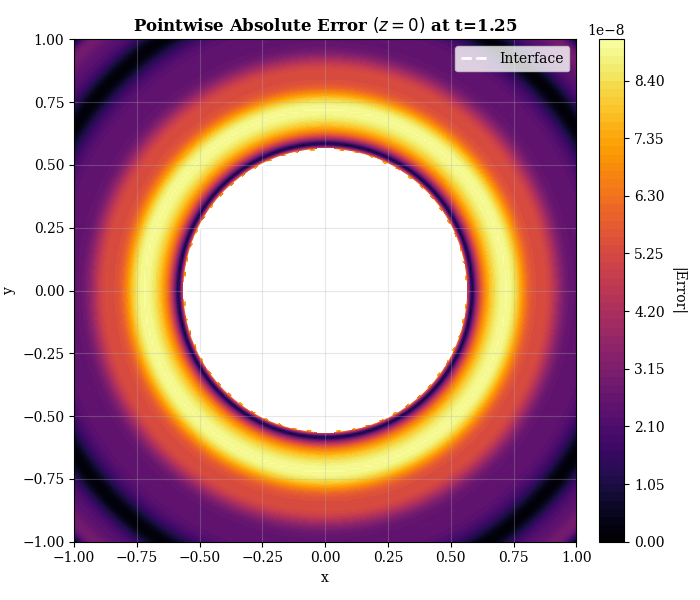} 
			\caption{}
			\label{fig:sub54}
		\end{subfigure}
		\begin{subfigure}[b]{0.3\textwidth}
			\centering
			\includegraphics[width=\textwidth]{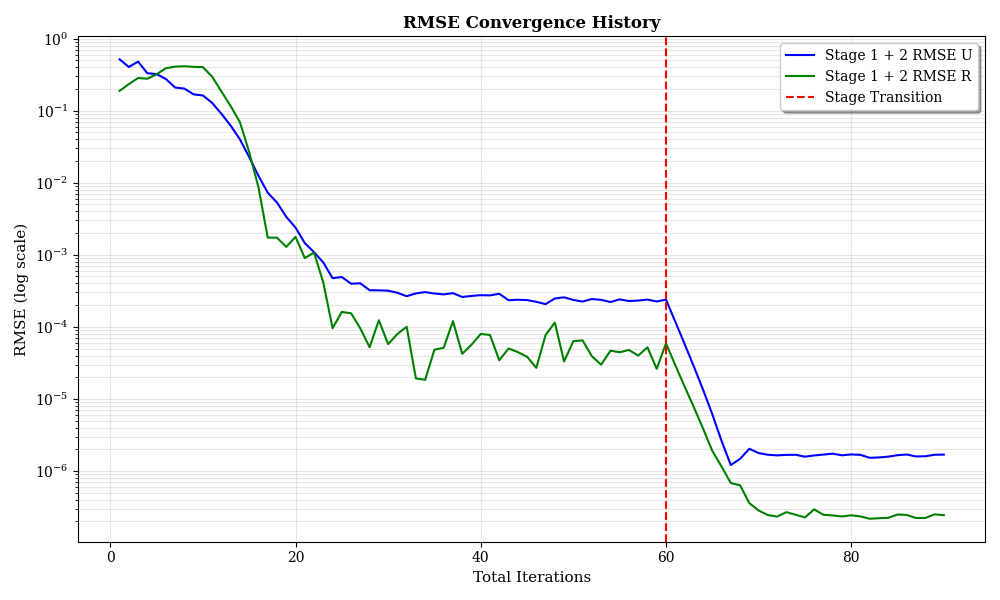}
			\caption{}
			\label{fig:sub51}
		\end{subfigure}
		\hfill
		\begin{subfigure}[b]{0.3\textwidth}
			\centering
			\includegraphics[width=\textwidth]{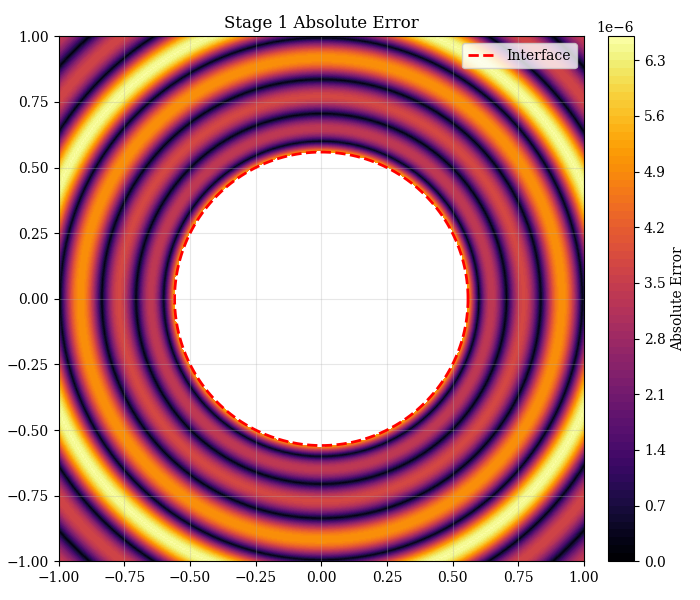}
			\caption{}
			\label{fig:sub58}
		\end{subfigure}
		\hfill
		\begin{subfigure}[b]{0.3\textwidth}
			\centering
			\includegraphics[width=\textwidth]{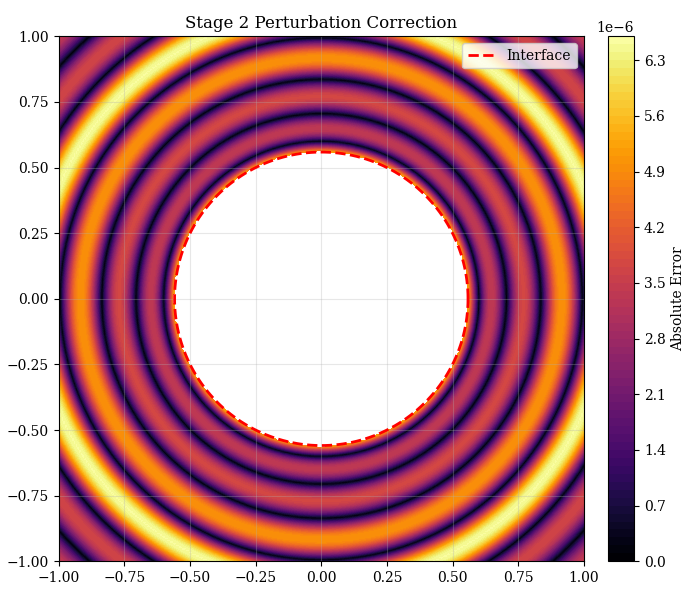}
			\caption{}
			\label{fig:sub59}
		\end{subfigure}
		
		\begin{subfigure}[b]{0.3\textwidth}
			\centering
			\includegraphics[width=1.1\textwidth]{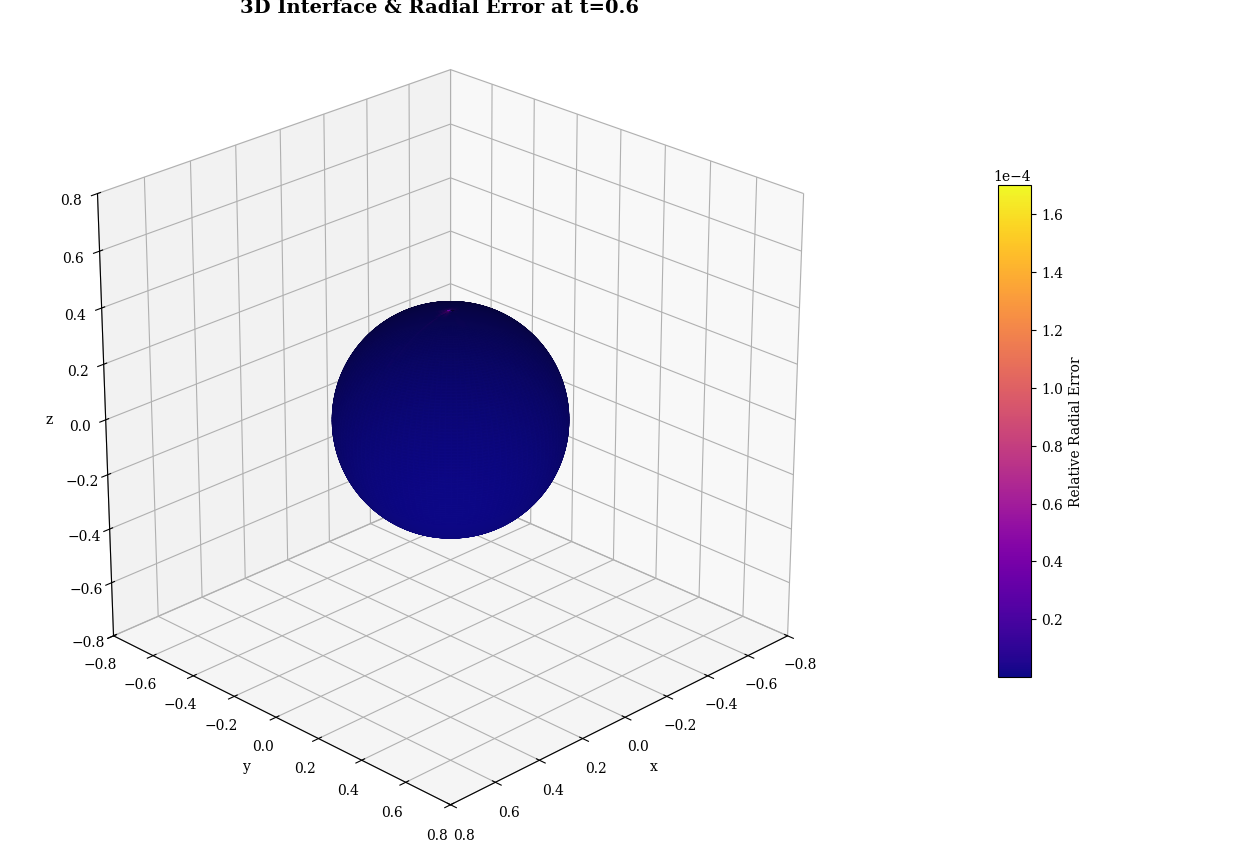} 
			\caption{}
			\label{fig:sub55}
		\end{subfigure}
		\hfill
		\begin{subfigure}[b]{0.3\textwidth}
			\centering
			\includegraphics[width=\textwidth]{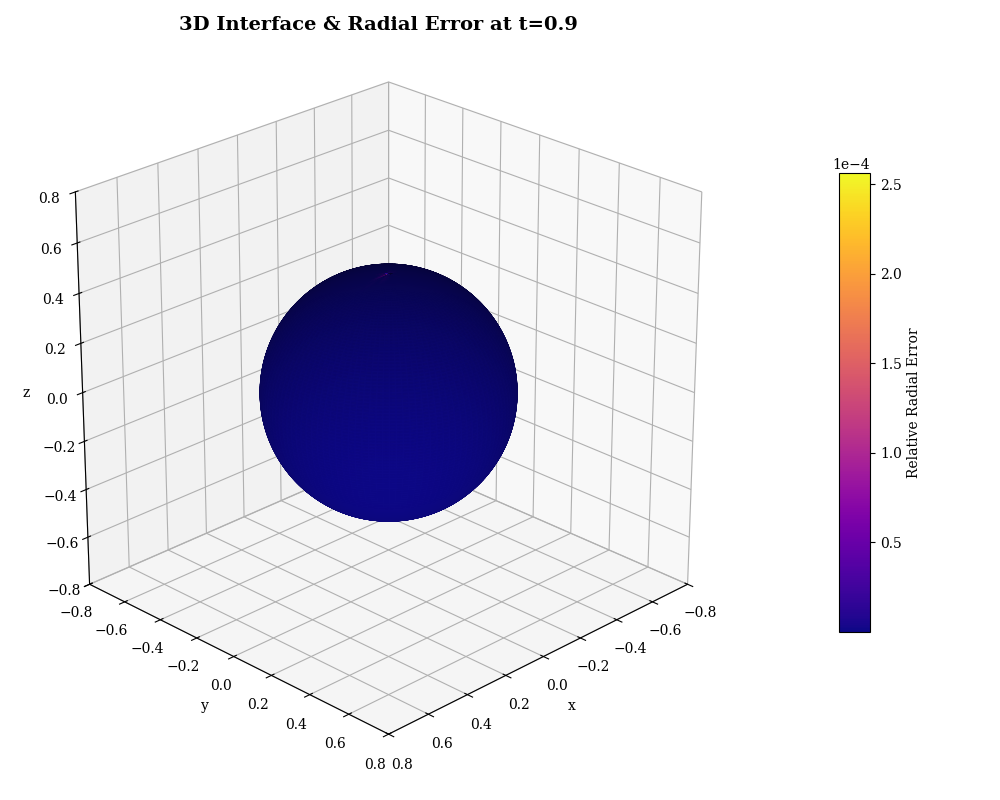} 
			\caption{}
			\label{fig:sub56}
		\end{subfigure}
		\hfill
		\begin{subfigure}[b]{0.3\textwidth}
			\centering
			\includegraphics[width=\textwidth]{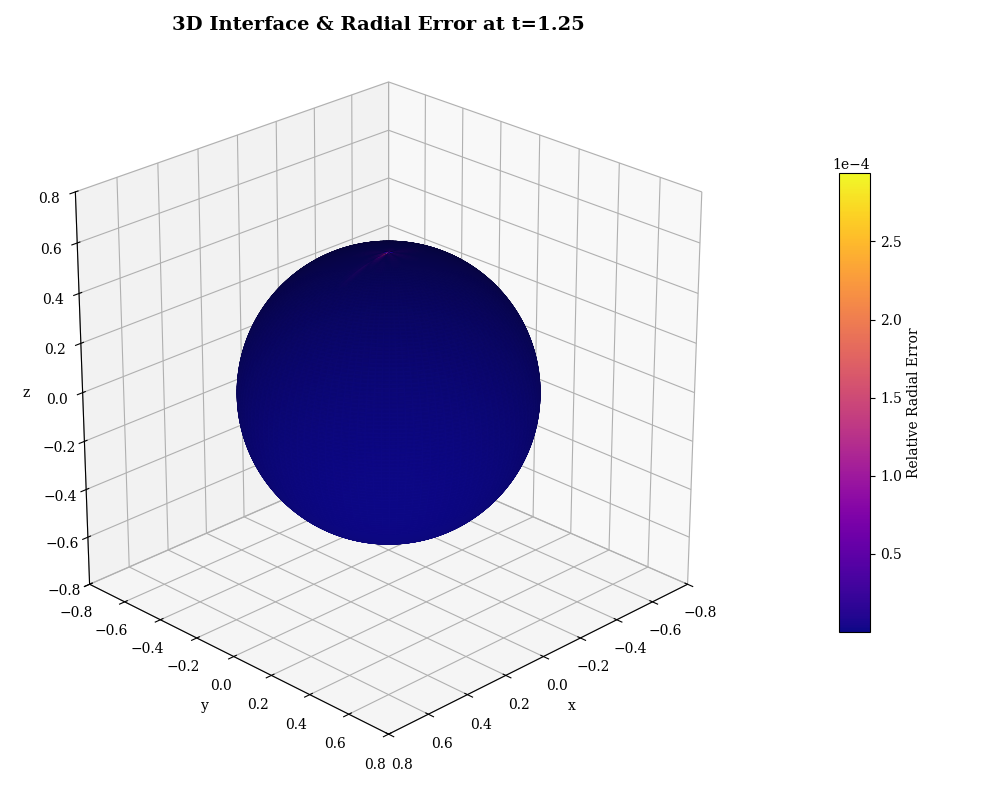}
			\caption{}
			\label{fig:sub57}
		\end{subfigure}

		\caption{Numerical results for the three-dimensional Frank-sphere problem: 
			(a) Point-wise absolute error of the temperature field at $t=0.6$; 
			(b) Point-wise absolute error of the temperature field at $t=0.9$;
			(c) Point-wise absolute error of the temperature field at $t=1.25$; 
			(d) RMSE iteration history in the training; 
			(e) Absolute error distribution of the liquid phase in Stage 1;
			(f) Absolute error of the liquid phase in the perturbation correction;
			(g) Visualization of the free boundary at $t=0.6$;
			(h) Visualization of the free boundary at $t=0.9$;
			(i) Visualization of the free boundary at $t=1.25$.}
		\label{fig:5}
	\end{figure}
	
	\section{Conclusion}
	\par This paper presents a PCELM-based method for the numerical solution of the Stefan problem. Based on the ELM method with a Gauss-Newton iteration for solving the Stefan problem, where the free boundary and temperature fields are simultaneously solved in the neural network, the perturbation correction for the residual from the initial step mostly captures the high frequency in the solution. It is remarkable that the proposed PCELM method improves the 2-6 orders of the magnitude in the numerical experiments. The comparison with other methods also demonstrates the accuracy advantage of our method in solving the Stefan problem.
	\par The universality and stability of this method make it applicable to a wide range of nonlinear partial differential equations. The inherent nonlinearity of these equations leads to non-convex optimization problems. In future research, we plan to extend this perturbation-correction strategy to a broader range of scientific computing fields, although these fields have complex non-convex optimization environments, high-precision solutions are still required. 
	
	\section*{Acknowledgements}
	The work was supported by the National Natural Science Foundation of China under Grant No. 12571431.

\end{document}